\documentclass{amsart}
\usepackage{graphicx}
\usepackage{color,amssymb}
\usepackage{stmaryrd}
\newcommand{\dt}{\partial_t}
\newcommand{\dx}{\partial_x}

\newcommand{\eps}{\varepsilon}
\newcommand{\jump}[1]{\llbracket#1\rrbracket}
\newcommand{\av}[1]{\langle#1\rangle}
\newcommand{\abs}[1]{\vert#1\vert}
\newcommand{\uq}{{\underline{q}}}
\newcommand{\uU}{{\underline{U}}}
\newcommand{\um}{{\underline{m}}}
\newcommand{\utheta}{{\underline{\theta}}}

\newcommand{\cD}{{\mathcal D}}
\newcommand{\cE}{{\mathcal E}}

\newcommand{\cG}{{\mathcal G}}

\newcommand{\cP}{{\mathcal P}}

\newcommand{\cR}{{\mathcal R}}
\newcommand{\cS}{{\mathcal S}}
\newcommand{\cT}{{\mathcal T}}

\newcommand{\dsp}{\displaystyle}

\newcommand{\mfm}{{\mathfrak m}}

\newcommand{\mfs}{{\mathfrak s}}
\newcommand{\mfe}{{\mathfrak e}}
\newcommand{\mfE}{{\mathfrak E}}

\newcommand{\RR}{{\mathbb R}}

\newcommand{\HH}{{\mathbb H}}
\newcommand{\BB}{{\mathbb B}}
\newcommand{\A}{{\mathbb A}}

\newtheorem{theorem}{Theorem}[section] 
\newtheorem{proposition}{Proposition}[section]
\newtheorem{lemma}{Lemma}[section]

\newtheorem{corollary}{Corollary}[section]

\newtheorem{notation}{Notation}
\newtheorem{remark}{Remark}[section]

\numberwithin{equation}{section}

\title[Waves Interacting with a Partially Immersed Obstacle]{Waves Interacting with a Partially Immersed Obstacle in the Boussinesq Regime}
\author{D. Bresch, D. Lannes, G. M\'etivier}
\address{Universit\'e Grenoble Alpes, Universit\'e Savoie Mont-Blanc et  CNRS UMR5127,
Batiment le chablais, 73376 Le Bourget du Lac Cedex, France}
\email{didier.bresch@univ-smb.fr}
\address{Institut de Math\'ematiques de Bordeaux\\ Universit\'e de Bordeaux et CNRS UMR 5251\\ 351 Cours de la Lib\'eration \\ 33405 Talence Cedex, France}
\email{David.Lannes@math.u-bordeaux.fr, Guy.Metivier@math.u-bordeaux.fr}
\thanks{D. B  is partially supported by the  ANR-18-CE40-0027 Singflows and the ANR project Fraise, and D. L is partially sypported by the ANR-18-CE40-0027 Singflows,  the Del Duca Fondation, the Conseil R\'egional d'Aquitaine 
and the ANR-17-CE40-0025 NABUCO.}

\begin{document}
\maketitle
\normalsize
\begin{small}
	\begin{center}
		{\bf Abstract}
	\end{center}
	This paper is devoted to the derivation and mathematical analysis of a wave-structure interaction problem which can be reduced to a transmission problem for a  Boussinesq system.  Initial boundary value 
problems and transmission problems in dimension d= 1 for $2\times2$ hyperbolic systems are well understood.  However, for many applications, and especially for the description of surface water waves, dispersive perturbations of hyperbolic systems must be considered. We consider here a configuration where the motion of the waves is governed by a Boussinesq system (a dispersive perturbation of the hyperbolic nonlinear shallow water equations), and in the presence of a fixed partially immersed obstacle. We shall insist on the differences and similarities with respect to the standard hyperbolic case,  and focus our attention on a new phenomenon, namely, the apparition of a dispersive boundary layer.  In order to obtain existence and uniform bounds on the solutions over the relevant time scale, a control of this dispersive boundary layer and of the oscillations in time it generates is necessary. This analysis leads to a new notion of compatibility condition that is shown to coincide with the standard hyperbolic compatibility conditions when the dispersive parameter is set to zero.   To the authors' knowledge, this is the first time that these phenomena (likely to play a central role in the analysis of initial boundary value problems for dispersive perturbations of hyperbolic systems) are exhibited.
	
	\bigskip
\noindent{\bf Keywords:} Wave-structure interaction, Boussinesq system, Free surface, Transmission problem,   Local well posedness, Dispersive boundary layer, Oscillations in time, Compatibility conditions. 
\end{small}

\section{Introduction}

\subsection{General setting}

 Free surface problems for various non-linear PDEs such as incompressible Euler and Navier-Stokes equations or reduced long-wave systems such as  the nonlinear shallow water equations, the Boussinesq and Serre-Green-Nagdhi  equations, have been strongly studied 
over the last decade: well-posedness, rigorous justification of asymptotic models, numerical simulations, etc.
 Recently, free surface interactions with floating or fixed structures have been addressed for instance in \cite{Lannes17} where a new formulation of the water-waves problem was proposed in order to take into account the presence of a floating body. In this formulation, the pressure exerted by the fluid on the partially immersed body appears as the Lagrange multiplier associated to the constraint that under the floating object, the surface of the water coincides with the bottom of the object. This method was also implemented in \cite{Lannes17} when the full water waves equations are replaced by simpler reduced asymptotic models such as the nonlinear shallow-water or Boussinesq equations. The resulting wave-structure models have been investigated mathematically when the fluid model is the nonlinear shallow water equations in the case of vertical lateral walls (\cite{Lannes17} in the one dimensional case and \cite{Bocchi} for  two-dimensional, radially symmetric, configurations) as well as in the more delicate case of non vertical walls where the free dynamics of the contact points must be investigated \cite{IL}. An  extension to the viscous nonlinear shallow water equations has also been recently derived and studied in \cite{Tucsnaketal}. On the numerical side, a discrete implementation of the above method was proposed in \cite{Lannes17} while a relaxation method was implemented in \cite{GPSMW1,GPSMW2} for one dimensional shallow water equations with a floating object and a roof. Note finally that the resulting equations present an incompressible-compressible structure as in congestion phenomena; the interested reader is referred to \cite{Perrin, Degond, BEG, PerrinZatorska, BreschRenardy} for instance.
 
 Mathematically speaking, the interactions of an immersed object with waves described by the one-dimensional shallow water equations can be reduced to an hyperbolic (possibly free boundary) transmission problem for which a general theory has been developed in the wake of the study of the stability of shock waves \cite{Majda,Metivier,BenzoniSerre} (see also \cite{IL} for a more specific general theory of one-dimensional free-boundary hyperbolic problems). If we want to consider more precise models for the propagations of the waves (e.g. Boussinesq or Serre-Green-Naghdi equations), the situation becomes more intricate because dispersive effects must be included and there is no general theory for transmission problems or even initial boundary value problems associated to such models. This is the situation we address in this paper where we consider waves described by a (dispersive) Boussinesq system interacting with a fixed partially immersed object with vertical lateral walls. The fact that the lateral walls are vertical simplify the analysis since the horizontal coordinates of the contact points between the object and the surface of the water are time independent. However, the discontinuity of the surface parametrization at the contact points makes the derivation of the model more complicated. This derivation is postponed to Section \ref{sectderiv} where we show that the wave-structure interaction problem under consideration can be reduced to the following transmission problem where $\zeta$ is the surface elevation above the rest state, $q$ the horizontal discharge, and where the bottom of the object is assumed to be the graph of a function $\zeta_{\rm w}$ above the interval $(-R,R)$ (see Figure \ref{fig}).
 \begin{figure}[h]
\begin{center}
\includegraphics[width=0.9\linewidth]{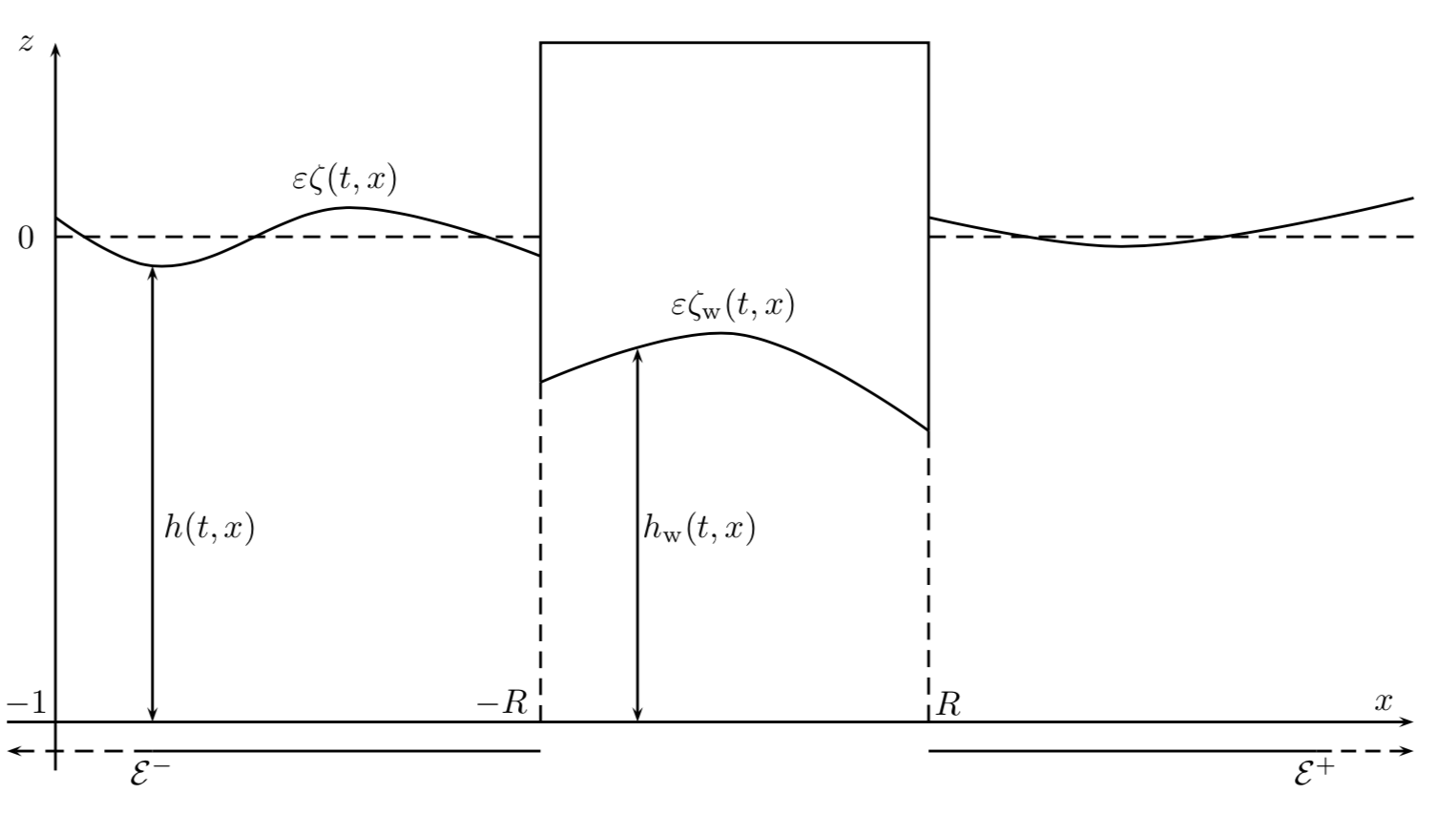}
\end{center}
\caption{A partially immersed obstacle}\label{fig}
\end{figure}
  In dimensionless form, this system reads
 \begin{equation}\label{3Bouss1_intro}
\begin{cases}
\dt \zeta +\dx q =0,\\
(1-\frac{1}{3}\mu \dx^2)\dt q +\eps \dx (q^2)+h \dx \zeta=0
\end{cases}
\quad \mbox{ on }\quad (-\infty,-R)\cup (R,+\infty);
\end{equation}
with $h=1+\eps \zeta$ and transmission conditions
\begin{align}\label{jump1_intro}
\jump{q}&=0,\\
\label{jump2_intro}
-\frac{\mu}{3}\dt \jump{\dx q}+\jump{\zeta+\varepsilon \frac{1}{2}\zeta^2}&=-\alpha \frac{d}{dt}\av{q},
\end{align}
where $\jump{q}$ and $\av{q}$ are defined as 
$$
\jump{q}=q(R)-q(-R)
\quad\mbox{ and }\quad
\av{q}=\frac{1}{2}\big( q(R)+q(-R)\big),
$$
and $ \alpha =\int_{-R}^R  1/h_{\rm w}$ where $h_{\rm w} = 1+\varepsilon \zeta_{\rm w}$. 
The system is completed by the initial condition
\begin{equation}\label{ICtransm_intro}
(\zeta,q)_{\vert_{t=0}}=(\zeta^{\rm in},q^{\rm in})
\end{equation}
where $(\zeta^{\rm in}, q^{\rm in})$ is given.
In this system $\eps$ is the so called amplitude parameter (the ratio of the typical amplitude of the waves over the depth at rest) and $\mu$ the shallowness parameter (the square of the ratio of the depth over the typical horizontal length). In the absence of floating object and in the Boussinesq regime (i.e. when $\mu \ll 1$ and $\eps \sim \mu$) the equations \eqref{3Bouss1_intro} are known to furnish a good approximation to the water waves equations \cite{Lannes_B}  for times of order $O(1/\eps)$. 

Our main objective here is to prove the local in time  well-posedness of the transmission problem \eqref{3Bouss1_intro}--\eqref{ICtransm_intro} on the same  $O(1/\eps)$ time scale. To our knowledge, this is the first time that such a result is proved for a dispersive perturbation of a hyperbolic system. Indeed, the theory of initial boundary value problems for hyperbolic systems has been intensely investigated \cite{Majda,Metivier,BenzoniSerre}, but even in the one-dimensional case, there are very few results for dispersive perturbation of such systems, despite their ubiquitous nature : Boussinesq systems for water waves \cite{Lannes_B}, internal waves \cite{BLS},  Euler-Korteweg system for liquid-vapour mixtures \cite{BDD}, elastic structures \cite{KT}, etc. At best, one can find local existence results for initial boundary value problems but on an existence time which shrinks to zero as the dispersive parameter $\mu$ goes to zero \cite{ADM,BonaChen,Xue,LW}, falling far below the relevant $O(1/\eps)$ time scale.

 In order to reach this time scale, it is necessary to analyze and control the \emph{dispersive boundary layer} that can be created at the boundary. In the particular case of the transmission problem \eqref{3Bouss1_intro}--\eqref{ICtransm_intro}, it is easy to construct local in time solutions (see Proposition \ref{prop1} below). It is striking that these solutions are smooth if the data are smooth \emph{without} having to impose compatibility conditions. This is in strong contrast with the hyperbolic case where it is well known that compatibility conditions of order $n-1$ are needed to obtain $H^n$ regularity. In our case, the dispersion automatically smoothes the solution by creating a \emph{dispersive boundary layer} that compensates the possible discontinuities of the derivatives at the corner $x=t=0$. In this boundary layer of typical size $O(\sqrt{\mu})$, the solution can behave quite wildly and standard techniques are unable to provide the necessary bounds to obtain an existence time independent of $\eps$ and $\mu$ (and a fortiori of order $O(1/\eps)$). One therefore needs to analyze precisely the dispersive boundary layer. By doing so, it is possible to derive a new kind of compatibility conditions (that are estimates rather than equations) that need to be imposed to control the dispersive boundary layer. Interestingly enough, these new conditions degenerate to the standard hyperbolic compatibility conditions when the dispersion parameter is set to zero. We believe that our approach, based on the analysis of the dispersive boundary layer, is of general interest and should play a central role in the yet to develop theory of initial boundary value problems for dispersive perturbations of hyperbolic systems.

\subsection{Organization of the paper}

The outline of the article is as follows:  In Section \ref{sectderiv}, we derive  the transmission problem \eqref{3Bouss1_intro}-\eqref{ICtransm_intro} from basic physical assumptions. In Section \ref{SMain} we perform a change of variables that linearizes  the transmission conditions and we rewrite the problem as an ODE which
is locally well-posed with a blow-up criterion related to a maximal existence time $T^*$ which may
depend of the small parameters $\varepsilon$ and $\delta$. The main objective is then to solve the system on  the relevant $O(1/\eps)$ time scale   showing 
a uniform bound of the quantity appearing in the blow-up criterion. Our main result states that this is the case provided that some compatibility conditions are satisfied. Due to the presence of dispersive terms in the equations, these compatibility conditions differ from the standard hyperbolic compatibility conditions; they are analyzed in details in Section \ref{ID}. We show in particular that if these new compatibility conditions are nonlocal, they can be approximated and replaced by local compatibility conditions that are easier to check. Section \ref{estimates} is then dedicated
to the uniform estimates of the time derivatives which rely on $L^2$ estimates for the linearized system and control of the commutators using modified Gagliardo-Nirenberg estimates in time and space that take into account the singularity of the multiplicative constants when the time interval is small.  Section \ref{dxD} is then dedicated to
estimates of $x-$derivatives. As in the hyperbolic case, we use the equations, but this is now much trickier. Because of the dispersive term, instead of getting explicit expressions $\partial_x^k\theta = \dots,$ in terms of time derivatives and lower order $x-$derivatives, we are led to solve equations of the form $(1+ \delta^2 \partial_t^2) \partial_x^k \theta = \dots$ and thus we have to control the rapid oscillations created by $(1+\delta^2\partial_t^2)$.  Finally, Section \ref{final} provides the proof of the main result, namely, the proof of Theorem \ref{theonew}, which is based on  the proof of a uniform bound of the quantity appearing in the blow-up criterion.
     
\subsection{Notations}\label{sectnota}
- Throughout this paper, we use the following notations for the jump and average of a function $f$ across the floating object,
$$
\jump{f}=f(R)-f(-R)
\quad\mbox{ and }\quad
\av{f}=\frac{1}{2}\big( f(R)+f(-R)\big),
$$
- We also denote
$$
\abs{x}_R=\begin{cases}
x-R & \mbox{ if } x>R,\\
-x-R& \mbox{ if } x<-R
\end{cases}.
$$
- We denote by $\cE=\cE^-\cup \cE^+$ the fluid domain, where
$$
\cE^-=(-\infty,-R)
\quad\mbox{ and }\quad
\cE^+=(R,\infty).
$$
- We denote $\HH=H^1(\cE)\times H^2(\cE)$, and more generally $\HH^n=H^{n+1}(\cE)\times H^{n+2}(\cE)$ for all $n\in {\mathbb N}$.\\
- We denote respectively by $R_0$ and $R_1$ the inverse of $(1-\delta^2 \dx^2)$ on $\cE$ with homogeneous Dirichlet and Neumann boundary data at $x=\pm$, that is, $R_0 f=u$ and $R_1 f=v$ with
$$
\begin{cases}
(1-\delta^2\dx^2) u= f,& \mbox{on }\cE\\
u_{\vert_{x=\pm R}}=0 & 
\end{cases}
\quad\mbox{ and }\quad
\begin{cases}
(1-\delta^2\dx^2) v= f,& \mbox{on }\cE\\
\dx v_{\vert_{x=\pm R}}=0 & 
\end{cases}.
$$
\section{Derivation of the model} \label{sectderiv}

\subsection{Basic equations}
We consider a wave-structure interaction problem consisting in describing the motion of waves at the surface of a one dimensional canal with a fixed floating obstacle. More precisely, we consider a shallow water configuration in which the waves are described, in dimensionless form, by the following Boussinesq system
\begin{equation}\label{2Bouss_base}
\begin{cases}
\dt \zeta +\dx q =0,\\
(1-\frac{1}{3}\mu \dx^2)\dt q +\eps \dx (q^2)+h \dx \zeta=- h \dx \underline{P},
\end{cases}
\end{equation}
where $\zeta$ is the surface elevation above the rest state, $h=1+\eps \zeta$ is the water depth, $q$ is the horizontal discharge (that is, the vertical integral of the horizontal component of the velocity field in the fluid domain), and $\underline{P}$ is the value of the pressure at the surface. The parameters $\eps$ and $\mu$ are respectively called the nonlinear and shallowness parameters and defined as
$$
\eps=\frac{a}{H_0} \quad\mbox{ and }\quad \mu=\frac{H_0^2}{L^2}
$$
where $a$ is the typical amplitude of the waves, $H_0$ the depth at rest, and $L$ the typical horizontal scale. The weakly nonlinear regime in which the Boussinesq system is valid (see \cite{Lannes_B} for instance for the derivation and justification of this Boussinesq model) is characterized by the relation
\begin{equation}\label{weakNL}
\eps\sim \mu.
\end{equation}

In this paper, we consider a fixed, partially immersed, object with vertical lateral walls located at $x=\pm R$ and assume that the bottom of the object can be parameterized by a function $\zeta_{\rm w}$ on $(-R,R)$. We shall refer to ${\mathcal I}=(-R,R)$ as the {\it interior} domain and to ${\mathcal E}=(-\infty,-R)\cup (R,\infty)$ as the {\it exterior}. The surface pressure is assumed to be given by the atmospheric pressure $P_{\rm atm}$ in the exterior domain, and by the (unknown) {\it interior pressure} $\underline{P}_{\rm i}$ on ${\mathcal I}$,
\begin{equation}\label{const1}
\underline{P}=P_{\rm atm}\quad \mbox{ on }\quad {\mathcal E}
\quad\mbox{ and }\quad
\underline{P}=\underline{P}_{\rm i}\quad \mbox{ on }\quad {\mathcal I}.
\end{equation}
The pressure is therefore constrained on ${\mathcal E}$ but not on ${\mathcal I}$, while this is the reverse for the surface elevation for which we impose
\begin{equation}\label{const2}
\zeta(t,x)=\zeta_{\rm w} (x)\quad \mbox{ on }\quad {\mathcal I},
\end{equation}
that is, the surface of the water coincides with the (fixed) bottom of the obstacle on ${\mathcal I}$.

Finally, transmission conditions are provided at the {\it contact points} $x=\pm R$ on the discharge and conservation of total energy is imposed,
\begin{align}\label{contact1}
q \mbox{ is continuous at }x=\pm R,\\
\label{contact2}
\mbox{There is conservation of energy for the wave-structure system};
\end{align}
the latter condition is made more precise in the next section, where we show that it yields a jump condition on the interior pressure that allows to close the equations. 

\subsection{Derivation of a jump condition for the interior pressure from energy conservation}

There are two different local conservation laws for the energy, one in the exterior region, and another one for the interior region.
\begin{itemize}
\item {\it  Local energy conservation in the exterior region}.
For the Boussinesq model \eqref{2Bouss_base} in the exterior region (i.e. with $-h\dx \underline{P}=0$), there is a local conservation of energy,
\begin{equation}\label{consNRJ}
\dt {\mathfrak e}_{\rm ext}+\dx {\mathfrak F}_{\rm ext}=0
\end{equation}
with
$$
{\mathfrak e}_{\rm _{ext}}=\frac{1}{2}\zeta^2+\frac{\varepsilon}{6}\zeta^3 +\frac{1}{2}q^2+\frac{\mu}{6}(\dx q)^2
\quad\mbox{ and }\quad
{\mathfrak F}_{\rm ext}=q\big[\zeta+\varepsilon \frac{2}{3}q^2+\varepsilon\frac{1}{2}\zeta^2-\frac{\mu}{3}\dx\dt q\big].
$$
\item {\it Local energy conservation in the interior region}.
Let us first remark that from the first equation of \eqref{2Bouss_base} and \eqref{const2}, one gets that $\dx q=0$ in the interior region. There exists therefore a function of time only $q_{\rm i}$ such that
$$
q(t,x)=q_{\rm i}(t)\quad\mbox{ on }\quad {\mathcal I}.
$$
The local conservation of energy reads
\begin{equation}\label{consNRJint}
\dt {\mathfrak e}_{\rm int}+\dx {\mathfrak F}_{\rm int}=0.
\end{equation}
with
$$
{\mathfrak e}_{\rm int}=\frac{1}{2}\zeta_{\rm w}^2 +\frac{1}{2h_{\rm w}}q_{\rm i}^2
\quad\mbox{ and }\quad
{\mathfrak F}_{\rm int}=q_{\rm i}\big[\zeta_{\rm w}+\underline{P}_{\rm i}\big]
$$
(recall that $\dx q_{\rm i}=0$) 
\end{itemize}
Since the object is fixed, the condition \eqref{contact2} is equivalent to saying that the total energy $E_{\rm tot}$ of the fluid should be constant, where
$$
E_{\rm tot}=\int_{\abs{x}<R}{\mathfrak e}_{\rm int}+\int_{\abs{x}>R}{\mathfrak e}_{\rm ext}.
$$
Time differentiating and using \eqref{consNRJ} and \eqref{consNRJint}, we impose therefore that
$$
0=-\jump{{\mathfrak F}_{\rm int}}+\jump{{\mathfrak F}_{\rm ext}}
$$
and using the continuity condition  \eqref{contact1} on $q$, this yields the following jump condition for the interior pressure,
\begin{equation}\label{condPi}
\jump{\zeta_{\rm w}+\underline{P}_{\rm i}}=\jump{\zeta+\varepsilon \frac{1}{2}\zeta^2-\frac{\mu}{3}\dx\dt q};
\end{equation}
setting $\mu=0$ in this relation, one recovers as expected the transmission condition obtained in \cite{Tucsnaketal} and \cite{Bocchi} for the nonlinear shallow water equations.

\subsection{Reformulation as a transmission problem}

We show in this section that the wave-structure interaction problem under consideration can be reduced to  the Boussinesq system \eqref{2Bouss_base} on the exterior domain ${\mathcal E}=(-\infty,-R)\cup (R,+\infty)$,
$$
\begin{cases}
\dt \zeta +\dx q =0,\\
(1-\frac{1}{3}\mu \dx^2)\dt q +\eps \dx (q^2)+h \dx \zeta=0
\end{cases}
\quad \mbox{ on }\quad {\mathcal E}
$$
together with transmission conditions relating the values of $\zeta$, $q$ (or derivatives of these quantities) at the contact points $x=\pm R$.
As noticed in the previous section, one has $\dx q=0$ in the interior region and $q(t,x)=q_{\rm i} (t)$  
on ${\mathcal I}$ for some function $q_{\rm i}$ depending only on time. From this and the continuity condition \eqref{contact1}, we infer
$$
\jump{q}=0 \quad \mbox{ and }\quad q_{\rm i}=\av{q},
$$
while the second equation of \eqref{2Bouss_base} implies that 
$$
\frac{1}{h_{\rm w}}\frac{d}{dt} \av{q}=-\dx \big( \underline{P}_{\rm i}+\zeta_{\rm w}\big).
$$
Integrating this relation on $(-R,R)$ gives  therefore
$$
\alpha \frac{d}{dt} \av{q}=-\jump{ \underline{P}_{\rm i}+\zeta_{\rm w}}\quad\mbox{ with }\quad
\alpha=\int_{-R}^R \frac{1}{h_{\rm w}}.
$$
Combining this with \eqref{condPi} provides a second transmission condition,
$$
\alpha \frac{d}{dt} \av{q} =-\jump{\zeta+\varepsilon \frac{1}{2}\zeta^2-\frac{\mu}{3}\dx\dt q}.
$$
We have therefore reduced the problem to the following transmission problem :
\begin{equation}\label{3Bouss1}
\begin{cases}
\dt \zeta +\dx q =0,\\
(1-\frac{1}{3}\mu \dx^2)\dt q +\eps \dx (q^2)+h \dx \zeta=0
\end{cases}
\quad \mbox{ on }\quad (-\infty,-R)\cup (R,+\infty);
\end{equation}
with transmission conditions
\begin{align}\label{jump1}
\jump{q}&=0,\\
\label{jump2}
-\frac{\mu}{3}\frac{d}{dt} \jump{\dx q}+\jump{\zeta+\varepsilon \frac{1}{2}\zeta^2}&=-\alpha \frac{d}{dt} \av{q},
\end{align}
where $\jump{\cdot}$ and $\av{\cdot}$ are defined as in  \S \ref{sectnota} and
$\alpha =\int_{-R}^R {1}/{h_{\rm w}}$ where $h_{\rm w} = 1+ \varepsilon \zeta_{\rm w}$. The system is completed by the initial condition
\begin{equation}\label{ICtransm}
(\zeta,q)_{\vert_{t=0}}=(\zeta^{\rm in},q^{\rm in}).
\end{equation}
The rest of this paper is devoted to the mathematical analysis of this transmission problem. Note that a similar problem without the dispersive terms was considered in \cite{Lannes17} in the $1D$ case,  in \cite{Bocchi} in the $2D$-radial case and in \cite{Tucsnaketal} with viscosity. As we shall see, the situation here is drastically different as the dispersive boundary layer plays a central role and requires the development of new techniques.

\section{Statement of the main result and sketch of the proof}
\label{SMain}

\subsection{Linearization of the transmission conditions}
Note that the limiting problem $\delta = 0$ is a hyperbolic transmission  problem, 
and thus we expect to recover the main features of such problems. In particular, the initial data must satisfy {\sl compatibility conditions } at the corners to get smooth solutions bounded on a uniform interval. But instead of being equations, these conditions are transformed into estimates, see below.

We also follow the general strategy of hyperbolic problems. The main part of the proof consists in proving uniform a priori estimates for smooth solutions. By construction, the system has a positive definite energy, providing good $L^2$-types estimates. The next step is to look for estimates for tangential  derivatives, that is here, time derivatives.  Then, one get estimates for the $\dx$ derivatives, using the equation. The system for the time derivative is much nicer when the boundary conditions are made linear, and this is why we introduce the new unknown  
$\theta$ (instead of the elevation $\zeta$) given by
$$
\theta=\zeta+\eps\frac{1}{2}\zeta^2
\quad \mbox{ or equivalently }
\zeta=\theta+\eps c(\theta)
\quad \mbox{ with }\quad
c(\theta)=-\frac{2\theta^2}{(1+\sqrt{1+2\eps\theta})^2}.
$$
Note that the equivalence (written above) comes from the fact that we consider uniformly bounded quantities with respect
to $\varepsilon$. Rewriting  the problem in terms of $\theta$ and $q$, one get
the following nonlinear system
\begin{equation}\label{Boussthetanl}
\begin{cases}
(1+\eps c'(\theta))\dt \theta+ \dx q=0,\\
[1- \delta^2\dx^2] \dt q+\varepsilon\dx (q^2) +\dx \theta=0
\end{cases}
\quad \mbox{ on }\quad {\mathcal E}
\end{equation}
with the \emph{linear}  transmission conditions
\begin{align}\label{Transs1}
\jump{q}&=0, \\
\label{Transs2}
-\delta^2\dt \jump{\dx q}+\jump{\theta}&=-\alpha \frac{d}{dt} \av{q}
\end{align}
and  the initial condition
\begin{equation}\label{initial1}
(\theta,q)\vert_{t=0} = (\theta^{\rm in}, q^{\rm in}),
\end{equation}
where $\theta^{in}= \zeta^{in} + \varepsilon \,(\zeta^{in})^2$. The main objective of
the paper is to prove  the existence and uniqueness of solutions of the above system
written in $(\theta,q)$  on a time interval $[0,T]$  such that $(\varepsilon+\delta^2)T$ 
is small enough under some compatibility conditions on the data. More precisely we
will prove Theorem \ref{theonew}.


\subsection{Reduction to an ODE}

 From now on, we restrict our attention to this new system \eqref{Boussthetanl}-\eqref{initial1}.  As a preliminary remark, we note that for $\delta > 0$, this system can be seen as an o.d.e. in a suitable Hilbert space. Introduce 
$R_0$ the inverse of $(1 - \delta^2 \dx^2)$ with Dirichlet boundary conditions on each side of $\mathcal E$. 
First note  that,  for initial data which satisfy $\jump{ q^{\rm in} } = 0$, the jump condition 
$ \jump{q} = 0$ is equivalent to $ \jump { \dt q} =0  $. Remark next that 
the second equation together with the jump condition $\jump { \dt q} =  0$  is equivalent to 
\begin{equation}\label{altq}
\dt q     =  -  R_0  \Gamma    + \sigma  e ^{ -  \frac{1}{\delta} \abs{x}_R } 
\end{equation}
with $\Gamma = \dx( \theta + \eps q^2)$, and necessarily 
$$\sigma = \frac{d}{dt}\av{q}.$$ 
This implies that 
$$
\delta^2 \dt \jump{\dx q} = - \delta^2 \jump{ \dx R_0 \Gamma } - 2 \delta \frac{d}{dt}\av{q} . 
$$ 
and therefore the second transmission condition \eqref{jump2} with $\theta= \zeta+ \varepsilon\zeta^2/2$ is equivalent to 
  $$
   - \delta^2 \jump{ \dx R_0 \Gamma } - 2 \delta \frac{d}{dt} \av{q} = \jump{\theta} +  \alpha \frac{d}{dt} \av{q}
  $$
that is
\begin{equation}
\label{dtqi}
\frac{d}{dt} \av{q} = -  \frac{1}{\alpha + 2 \delta} \Big( \delta^2 \jump{ \dx R_0 \Gamma  }  + \jump{ \theta } \Big). 
\end{equation} 
\begin{remark}
The fact that $\frac{d}{dt} \av{q}$, which is the coefficient of the exponentially decaying term in \eqref{altq}, is given explicitly in terms of $q$ and $\theta$ is crucial for the ODE formulation of Proposition \ref{lemNLEDO}.
\end{remark}
Hence we have proved the following result, where we recall that $\HH=H^1 ( \mathcal E) \times H^2  ( \mathcal E)$.

\begin{proposition}\label{lemNLEDO}
For  $(\theta, q) \in C^1 ([0, T] ; \HH)$ such  that 
 $\inf_{[0,T]\times \cE}\{1+\eps c'(\theta)\}>0$ and 
$\jump{ q_{ | t = 0} } = 0$, 
the system \eqref{Boussthetanl}--\eqref{Transs2}  is equivalent to
\begin{equation}
\label{nlode}
\dt   U    =  \mathcal L (U)    := 
\\
    \left(\begin{array}{c}
\dsp - \Phi  \\
\dsp -  \cR \big( \Gamma , \jump{\theta} \big) 
\end{array}\right).
\end{equation} 
with 
\begin{equation}
\label{nlrhs}
\Phi =  \frac{1}{1+\eps c'(\theta)}\dx q, \qquad \Gamma = \dx (\theta + \eps q^2 ). 
\end{equation}
and
\begin{equation}
\label{defCR}
\cR (\Gamma, \rho)  = R_0 \Gamma + 
\frac{1}{\alpha+2\delta}\big(\delta^2 \jump{\dx R_0  \Gamma }+  \rho \big)
e ^{ -  \frac{1}{\delta} \abs{x}_R }. 
\end{equation} 
\end{proposition} 
The smoothing properties of 
$R_0$ imply that the mapping $\mathcal L$   is smooth from $\HH^n=H^{n+1} (\mathcal E)\times H^{n+2}(\mathcal E)$  to itself  for $n \ge 0$, so that the system \eqref{nlode} can be solved in this space, on an interval of time which depends on 
$\delta$, see Proposition~\ref{prop1} below.

\begin{remark} \textup{In sharp contrast, the situation when 
$\delta = 0 $ is quite different. In this case, the equations can be written 
\begin{equation}
\label{nlhyp}
\dt  U = \mathcal L_0 (U)    
:=  \left(\begin{array}{c}
\dsp -\frac{1}{1+\eps c'(\theta)}\dx q \\
\dsp -   \dx  ( \theta + \eps q^2 ) 
\end{array}\right)
\end{equation} 
but  $\mathcal L_0$  is not anymore continuous on a fixed Sobolev space. More importantly, the boundary conditions 
\begin{equation}
\label{hypbc}  
\jump{q} = 0, \qquad \jump{\theta} = - \alpha \dt \av{q}
\end{equation}
are \emph{not} propagated by the equations and, as usual for hyperbolic problems,  the initial data  must satisfy compatibility conditions to generate smooth solutions. This shows that \eqref{nlode}  is a truly singular perturbation of \eqref{nlhyp}, not only because one passes from bounded to unbounded operators, but also because of the boundary conditions,  which are included in \eqref{nlode} and not in \eqref{nlhyp}. }
\end{remark}

\medbreak
 
The discussion above  shows that the problem is reasonably well posed, but does not answer our objective : our goal  is to solve  \eqref{Boussthetanl}-\eqref{initial1} on an interval of time  independent of  $\delta$ and
$\eps$, and even more, of size $O(1/\eps)$. The first step is to use Cauchy Lipschitz theorem to prove local existence and derive a blow up criterion. We recall that $\HH^n  
= H^{n+1} (\mathcal E)\times H^{n+2}(\mathcal E)$.
\begin{proposition} 
\label{prop1} For $n \ge 0$, consider  initial data $(\theta^{\rm in}, q^{\rm in}) \in \HH^n $  satisfying $\jump{q^{\rm in} } = 0$ and 
 $\inf  \{ 1+\eps c'(\theta^{\rm in} )\}>0$. Then for all $\eps \in [0,1]$ and
  $\delta > 0$, there is $T > 0$ such that 
the system \eqref{Boussthetanl}--\eqref{initial1} has a unique solution in $C^1([0,T[ ; \HH^n)$, which in addition belongs to $C^\infty([0,T[ ; \HH^n)$. Moreover, if $T^*$ denotes the maximal existence time and $T^*<\infty$, one has
\begin{equation}
\label{maxT}
\lim _{ T \to TÅ¡*}   \big\| \theta , q, \dx q,  1/ (1+ \eps c'(\theta) )\big\|_{L^\infty ([0, T] \times \mathcal E)}  = + \infty . 
\end{equation}

\end{proposition} 

\begin{proof}
Let $\mathcal O$ denote the open subset of $\HH$ of the $U = (\theta, q)$ such that 
 $\inf  \{ 1+\eps c'(\theta )\}>0$.  Then  $\Phi (U) = \dx q / (1+\eps c'(\theta))$ is a smooth mapping from
 $\mathcal O$ to  $H^{n+1}  (\mathcal  E)$ and 
 $  \Gamma (U) = \dx (\theta + \eps q^2 )$ is a smooth mapping from $\HH$ to $H^n(\mathcal E)$. For $\delta > 0$, $R_0 $ maps $H^n(\mathcal E)$ to $H^{n+2} (\mathcal E)$, so that the right hand side $\mathcal L (U)$ of \eqref{nlode} 
is a smooth mapping from $\mathcal O$ to $\HH$ and the local existence of solutions of \eqref{nlode} 
$U \in C^\infty ([0, T) ; \mathcal O)$ follows for $T>0$ small enough. 
   Next,  when $f$ is smooth with $f(0) = 0$,  one has 
\begin{equation}
\big\| f(u) \big\|_{H^k(\mathcal E)}  \le C( \| u \|_{L^\infty} )  \big\| u  \big\|_{H^k(\mathcal E)} 
\end{equation}
where $C(.)$ is a continuous function on $\RR$. 
This implies that 
for $U \in \mathcal O$ one has 
\begin{equation*}
\big\| \Phi (U) \big\|_{H^{n+1} (\mathcal E)}  \le C (\mfm_0 (U) )  \big\|  U \big\|_{\HH^n} , \quad
\big\| \Gamma (U) \big\|_{H^{n} (\mathcal E)}  \le C (\mfm_0 (U) )  \big\|  U \big\|_{\HH^n} ,
\end{equation*}
and therefore
\begin{equation*}
\big\| \mathcal L (U) \big\|_{\HH^n}  \le C (\mfm_0 (U) )  \big\|  U \big\|_{\HH^n} ,
\end{equation*}
 where 
 \begin{equation}
\label{mfm0}
\mfm_0 (U) =  \big\| \theta ,  q, \dx q, ,  1/ (1+ \eps c'(\theta) )\big\|_{L^\infty (\mathcal E)}. 
\end{equation} 
and  $C (\mfm)$  depends only on $\mfm$. Thus the solution  satisfies
 $$\displaystyle
  \big\|  U (t)  \big\|_{\HH^n}  \le  e^{ \int_0^t  \mfm_0(U)(s) ds }\  \big\|  U (0)  \big\|_{\HH^n} 
 $$
 implying that it  be continued as long as 
 $\mfm (U(t))$ remains bounded. The second part of the proposition follows. 
\end{proof}

\subsection{Compatibility conditions}

Proposition \ref{prop1} shows that in order to prove existence on a time interval $[0,T]$ (with $T$ independent of $\delta $ and a fortiori of size $O(1/\eps)$), it is sufficient to prove a priori estimates which imply that $\mfm_0 (U(t))$ remains uniformly bounded on $[0, T] $.  For that, we follow the lines of the analysis of hyperbolic equations, based on energy estimates. The $L^2$-type estimates for $U = (\theta, q)$ are easy, and differentiating the equations in time, one obtains uniform estimates  for $U_j  = (\theta_j , q_j) =  (\dt^j \theta, \dt^j q)$ in terms of ${U_j}_{\vert_{t=0}}$, provided that the commutator terms can be controlled -- see Proposition \ref{Estdtj} below. However, such a control is useless if one cannot prove that the energy of the initial value of the time derivatives, namely, ${U_j}_{\vert_{t=0}}$, is uniformly controlled in terms of standard Sobolev norms of the initial data $U_{\vert_{t=0}}$. This is the issue addressed in this section.

\medbreak

We therefore seek to give conditions which ensure uniform bounds for the initial values $U^{\rm in}_j$  of the $U_j$.  By  \eqref{nlode}, one can compute them inductively. Namely one has,  when 
$\delta > 0$, 
\begin{equation}
\label{inducid}
\left\{\begin{aligned}
&  
\theta_{j+1}^{\rm in} =  - \Phi^{\rm in}_{j}  \\
& q_{j+1}^{\rm in}  = -  \cR \big( \Gamma_j^{\rm in}, \jump{\theta_{j}^{\rm in} }\big)
\end {aligned}  \right.
\end{equation} 
where 
\begin{equation}
\cR (\Gamma_j^{\rm in}, \jump{\theta_{j}^{\rm in} })  = R_0 \Gamma_j^{\rm in} + 
\frac{1}{\alpha+2\delta}\big(\delta^2 \jump{\dx R_0  \Gamma_j^{\rm in} }+ \jump{\theta_{j}^{\rm in}}   \big)
e ^{ -  \frac{1}{\delta} \abs{x}_R }
\end{equation} 
and
\begin{equation}
\label{PhiGamma} 
\Phi_j^{\rm in}=  \frac{1}{1+\eps c'(\theta_j^{\rm in})}\dx q_j^{\rm in} , 
        \qquad \Gamma_j^{\rm in}  =  \dx  ( \theta_{j}^{\rm in}  + \eps (q_j^{\rm in})^2 ),
\end{equation}
  using systematically the notations $f_j = \dt^j $, $f^{\rm in} = f_{| t = 0} $, $f_j^{\rm in} = \dt^j f _{| t = 0} $. 
Indeed,  
  $\Phi^{\rm in}_{j} $  and $ \Gamma_{j}^{\rm in}$ are non linear functions 
  of $(\theta_k^{\rm in}, q_k^{\rm in})$   $(\dx \theta_k^{\rm in}, \dx q_k^{\rm in})$ for $k \le j$, so that
\eqref{inducid} defines inductively $U^{\rm in}_j $ for all $j$ in $ \HH^n$ if 
$U^{\rm in}_0 \in \HH^n$. 
  
The difficulty is that the relations \eqref{inducid} do not provide a uniform control (with respect to $\delta$) of the space derivatives $\dx^k U_{j+1}^{\rm in}$ in terms of space derivatives of the $U_j$. Indeed, it follows from \eqref{inducid} and the definition \eqref{defCR} of $\cR$ that
$$
\dx^k q_{j+1}^{\rm in}=- \dx^k R_0 \Gamma_j^{\rm in}- \frac{1}{\alpha+2\delta}(\delta^2\jump{\dx R_0\Gamma_j}+\jump{\theta_j})\frac{d^k}{dx^k}\big(e^{-\frac{1}{\delta}\abs{x}_R}\big),
$$
and it appears that both terms in the right-hand-side are of size $O(\delta^{-k})$ (see \S \ref{sect41} for details). The only way one can expect a uniform control of $\dx^k q_{j+1}^{\rm in}$ is that these two singular terms cancel one another. This is the case provided that the following {\it compatibility conditions} are satisfied, for some $M>0$ and all $\delta\in (0,1]$,
\begin{equation}
\label{cc1new} 
\begin{cases}
\big\lvert \jump{q_{j+1}^{\rm in}} \big\rvert  \le M \delta^{n-j-1/2} \\
\big\lvert  \alpha \av{q_{j+1}^{\rm in}} +\jump{\theta_{j}^{\rm in}} -\delta^2 \jump{\dx q_{j+1}^{\rm in}} \big\rvert   \le M \delta^{n-j-1/2} 
\end{cases}
 \qquad  \mathrm{for} \ \  0 \le j \le n -1
\end{equation}
(roughly speaking, this means that the transmission conditions \eqref{Transs1} and \eqref{Transs1} are approximately satisfied by the $U_j^{\rm in}$ up to $j=n-1$).   
\begin{remark} \textup{Note that the $(\theta_{j,k} , q_{j,k} )$  are given by nonlinear functionals 
of $U^{\rm in}=(\theta_0^{\rm in}, q_0^{\rm in})$ involving $R_0$ and space derivatives, of total order at most $j$. Therefore the conditions above are assumptions bearing only  on the initials data $U^{\rm in}$.  }
\end{remark}

Under such conditions, it is possible to control the $U_j^{\rm in}$ in Sobolev spaces, as shown in the following proposition whose proof is postponed to Section \ref{ID} for the sake of clarity.
\begin{proposition} 
\label{estID}
Given $n\in {\mathbb N}$ and $M>0$, there is a constant $C$ such that  
for all initial data $(\theta_0^{in}, q_0^{in}) \in \HH^n$ and parameters 
$\eps$ in $[0, 1]$  and $\delta \in (0,1]$ satisfying \begin{equation}
\label{H0}
\jump{q_0^{\rm in}}=0, \quad \| \theta_0^{\rm in} \|_{H^{n+1} (\mathcal E)}  \le M ,
 \quad \|  ( q_0^{\rm in}, \delta \dx q_0^{\rm in} )  \|_{H^{n+1} (\mathcal E)}  \le M , 
\end{equation}
\begin{equation}\label{H1}
 1+\eps c'(\theta_0^{\rm in }) \geq M^{-1} ,
\end{equation}
and the conditions \eqref{cc1new} for $j < n$,    one has  
 \begin{equation}
 \label{estid}
 \big\| ( \theta_j^{\rm in},  q_j^{\rm in}, \delta \dx q_j^{\rm in})   \big\|_{H^{ n+1- j} (\mathcal E)} \le C 
 \qquad for \ \  0 \le j \le n +1. 
 \end{equation}
\end{proposition}

%
%
%
%
%

\subsection{The main theorem}

This being settled, our main result is the following. We recall that $\HH^n=H^{n+1}(\cE)\times H^{n+2}(\cE)$.
\begin{theorem} \label{theonew}
 Let $n\ge 5$. Given $M>0$, there is $\tau>0$ such that   for all initial data $(\theta_0^{\rm in}, q_0^{\rm in}) \in \HH^n $ and parameters  $\eps \in [0,1]$ and $\delta \in (0, 1]$  satisfying \eqref{H0} and \eqref{H1}, and the compatibility 
conditions \eqref{cc1new},  there is a unique solution    $U=(\theta,q) \in {\mathcal C}^1([0,T]; \HH^n)$ of \eqref{Boussthetanl}--\eqref{initial1},   with 
 $ T =  \tau / (\varepsilon+\delta^2)$.
\end{theorem}
\begin{remark}\textup{
Recalling that $\delta^2=\frac{1}{3}\mu$ and the assumption \eqref{weakNL} of weak nonlinearity, namely, $\eps \sim \mu$, the theorem provides an existence time which is $O(\frac{1}{\eps+\delta^2})=O(\frac{1}{\eps})$ which is the same as for the initial value (Cauchy) problem \cite{SautXu,Burtea}.}
\end{remark}

A drawback of Theorem \ref{theonew} is that the compatibility conditions \eqref{cc1new} are not easy to check, as the construction of the $U_j^{\rm in}$ involve the nonlocal operator $R_0$ through the operator $\cR$ in \eqref{inducid}. For instance, it is not clear to assess wether smooth initial data compactly supported away from the boundary satisfy the compatibility conditions \eqref{cc1new}. Taking advantage of the fact that the compatibility conditions \eqref{cc1new} are {\it estimates} rather than equations, we derive here a set of approximate compatibility conditions that do not involve nonlocal operator (and therefore much easier to check) that are sufficient to obtain the result of Theorem \ref{theonew}.

We start by noticing that the second equation in \eqref{inducid} can be equivalently written
$$
q_{j+1}^{\rm in}=-(1-\delta^2 \dx^2)^{-1}\Gamma_j
$$
where the inverse operator is associated to the boundary conditions 
$$q_{j+1}^{\rm in}\,_{\vert_{x=\pm R}}=\frac{1}{\alpha+2\delta}\big(\delta^2 \jump{\dx R_0 \Gamma_j}+\jump{\theta_j}\big).$$
A very na\"{\i}ve approximation of this formula is to replace the inverse by its Neumann expansion,
$$
q_{j+1}^{\rm in}\sim - \sum_{2l<n-j}  \delta^{2l}\dx^{2l}\Gamma_j.
$$
Replacing the second equation in \eqref{inducid} by this approximation leads us to approximate  $\dx^j U^{\rm in}_j$ by the $\widehat{U}^{\rm in}_{j,k}$ defined through the induction relation
\begin{equation}
\label{inducidapp}
\widehat{U}_{0,k}=\dx^k U^{\rm in}
\quad \mbox{ and }\quad
\left\{\begin{aligned}
&  
\widehat{\theta}_{j+1,k}^{\rm in} =  - \widehat{\Phi}^{\rm in}_{j,l}  \\
& \widehat{q}_{j+1,k}^{\rm in}  = -   \sum_{2l<n-k-j}  \delta^{2l}\widehat{\Gamma}_{j,2l+k}
\end {aligned}  \right. 
, \qquad j+k< n,
\end{equation} 
with $\widehat{\Phi}^{\rm in}_{j,k}$ and $\widehat{\Gamma}_{j,k}$ defined as $\dx^k\Phi_j$ and $\dx^k\Gamma_j$ but in terms of the $\widehat{U}_{j,k}$  rather than the $\dx^k U_j$. 
It is then natural to define the following approximate compatibility conditions: for some $M>0$ and all $\delta\in [0,1]$,
\begin{equation}
\label{cc1newapp} 
\begin{cases}
\big\lvert \jump{\widehat{q}_{j+1}^{\rm in}} \big\rvert  \le M \delta^{n-j-1/2} \\
\big\lvert  \alpha \av{\widehat{q}_{j+1,0}^{\rm in}} +\jump{\widehat{\theta}_{j,0}^{\rm in}} -\delta^2 \jump{\widehat{q}_{j+1,1}^{\rm in}} \big\rvert   \le M \delta^{n-j-1/2} 
\end{cases}
 \qquad  \mathrm{for} \ \  0 \le j \le n -1.
\end{equation}
\begin{remark} \textup{Note that $(\widehat{\theta}_{j}^{\rm in} , \widehat{q}_{j}^{\rm in} )$  are given by nonlinear functionals 
of $U^{\rm in}=(\theta_0^{\rm in}, q_0^{\rm in})$ involving space derivatives, of total order at most $n$. Therefore the conditions above are assumptions bearing only  on the initials data $U^{\rm in}$. The difference with the compatibility condition \eqref{cc1new} is that they do not involve nonlocal operators and just bear on the Taylor expansion of the initial data at the boundaries $x=\pm R$. They are therefore much easier to check; it is for instance trivial to verify that they are satisfied by smooth data compactly located away from the boundaries. }
\end{remark}

The fact that the approximate compatibility conditions are sufficient to keep the results of Proposition \ref{estID} and Theorem \ref{theonew} is not obvious, and the proof of the following corollary is left to Section \ref{ID} for the sake of clarity.
\begin{corollary}\label{coromain}
Under the same assumptions but replacing the compatibility condition \eqref{cc1new} by its approximation \eqref{cc1newapp}, the results of Proposition \ref{estID} and Theorem \ref{theonew} remain valid.
\end{corollary}
\begin{remark}
Contrary to Theorem \ref{theonew}, Corollary \ref{coromain} remains valid when $\delta=0$. In this case, the transmission problem \eqref{Boussthetanl}-\eqref{initial1} is hyperbolic and the approximate compatibility conditions \eqref{cc1newapp} are the standard hyperbolic compatibility conditions.
\end{remark}

\subsection{Outline of the proof}

 The outline of the paper is as follows:  Section~\ref{ID}  is devoted to the derivation, analysis  and approximation of the compatibility conditions. In section~\ref{estimates}  we prove \emph{a priori} uniform estimates for the time derivatives 
$U_j=\dt^j U$. This requires $L^2$ estimates for the linearized system (linear stability) and a  control of the commutators for which we use a generalization of the space-time Gagliardo-Nirenberg estimates that makes explicit the singular dependence of the constants when the time interval is small.   \\
In Section~\ref{dxD} we prove \emph{a priori} uniform estimates for spatial derivatives.  As in the hyperbolic case, we use the equations, but this is now much trickier.  Because of the dispersive term, instead of getting explicit expressions  $\dx^k \theta  = \dots  $  in terms of time derivatives and lower order $x$-derivatives, we are led to solve  equations of the form
$$
(1 + \delta^2 \dt^2) \dx^k \theta = \dots 
$$
and thus we have to control the rapid oscillations created by $(1 + \delta^2 \dt^2)$. Finally, the proof of the main Theorem~\ref{theonew} and of Corollary \ref{coromain} in done in Section~\ref{final}. It is based on a control of the blow up criterion provided by Proposition \ref{prop1}.


\section{Compatible initial data}
\label{ID} 

The goal of this section is to prove the uniform estimates of Proposition~\ref{estID} under the compatibility conditions \eqref{cc1new} and to show, as claimed in Corollary \ref{coromain} that these uniform estimates remain true under the approximate compatibility conditions   \eqref{cc1new}. \\
Throughout this section  $\delta \in (0, 1]$ and we think of it as beeing  small. 


\subsection{Analysis of the mapping $\cR$ } \label{sect41}

The key point for the derivation of the compatibility conditions is the analysis of the second, nonlocal, equation in the induction relation \eqref{inducid} that is used to construct the $U_j^{\rm in}$. We are therefore led to study  the operator $\cR$ defined in \eqref{defCR} : 
\begin{equation}
\label{qind}
 \cR: (f, \rho )  \quad \mapsto \quad   R_0 f  + \frac{1}{\alpha + 2 \delta} \big(\rho +  \delta^2 \jump{ \dx R_0 f }  \big)  e^{ - \frac{1}{\delta} |x|_R}.
\end{equation}
This operator is well defined from   $ ( H^{k} (\cE)\times \RR)  $ to  $H^{k}(\cE)$ for $k \ge 0$. We look for estimates which make explicit the dependence on the parameter $\delta$. 
When $k = 0$, we note that  $R_0$ ,  $\delta  \dx R_0 $   and $\delta^2  \dx^2 R_0 $  are uniformly bounded in  $L^2$. In particular,   
$$
\big|  \delta^2 \jump{ \dx R_0 f }  \big|^2   \le \| \delta ^2  \dx R_0 f \|_{L^2} 
 \| \delta ^2  \dx^2 R_0 f \|_{L^2}  \le C \delta \| f\|_{L^2}^2 
$$
 so that, with $q=\cR (f,\rho)$,
\begin{equation}
\label{est25} 
\| (q, \delta \dx q )\|_{L^2} \le C( \| f \|_{L^2} +  \delta^{1/2} | \rho  | ) . 
\end{equation}
When $k > 0$, the difficulty is that $\dx$ and $R_0$ do not commute and that 
 the layer $ e^{ - \frac{1}{\delta} |x|_R}$
is not uniformly bounded in $H^1$. In this section, we show that compatibility conditions are needed in order to obtain uniform higher order Sobolev estimates on $q=\cR(f,\rho)$.

\subsubsection{Higher order estimates of $R_0$}

As previously said, the $H^k$-norms of $R_0 f$ have a singular dependence on $\delta$. We make here this singular dependence explicit.  Let $R_1$ denote the inverse of $(1-\delta^2 \partial_x^2)$ on $\cE$ with homogeneous Neumann conditions at $x=\pm R$. 
\begin{proposition}
\label{lemtr1}
For   $k  >  0  $ and $f\in H^k(\cE)$, one has 
$$
\dx^k R_0 f  = R_{\iota} \dx^k f   -   ( \mp 1)^k   \delta^{-k} 
(\mathcal D_k^\pm   f ) e^{  - \frac{1}{\delta} | x|_R}   \qquad  \mbox{on }\quad \cE_\pm
$$with 
\begin{itemize}
\item   $\iota = 0$ when $k$ is even  
\item  $\iota = 1$ when $k$ is odd, 
\end{itemize}
and
$$
\mathcal D_k^\pm f  =  \sum_{ l < k/2} (\delta \dx)^{2l} f _{|  x = \pm R} . 
$$
\end{proposition}

\begin{proof}
For all $f\in H^1(\cE)$, the following identities hold,
\begin{align*}
\dx R_0 f&=R_1\dx f  \  \pm \ \delta^{-1} f_{\vert_{x=\pm R}}
  e^{-\delta^{-1}\vert x\vert_R} &\mbox{on }\quad \cE_\pm,\\
\dx R_1 f&=R_0\dx f &\mbox{on }\quad \cE_\pm . 
\end{align*}
For the first identity, just notice that if  $u = R_0 f$,  then $v=\dx u -c e^{-\delta^{-1}\vert x\vert_R}$ solves $(1-\delta^2 \dx^2)v = \dx f$ on $\cE$ with boundary condition 
$$
\delta^2 \dx v_{\vert_{x=\pm R}}=  \delta^2 \dx^2 u_{| x = \pm R} \pm \delta c = 
-f_{\vert_{x=\pm R}}\pm \delta c. 
$$
Thus   $\dx v_{\vert_{x=\pm R}}=0$ and  $v=R_1\dx f$  if $c=\pm \delta^{-1}f_{\vert_{x=\pm R}}$. 
For the second identity, if $u = R_1 f$,  then $v= \dx u$ solves $(1-\delta^2 \dx^2)v = \dx f$ with boundary condition $v_{\vert_{x=\pm R}}=0$, so that $v=R_0\dx f$.

\medskip

\noindent Hence, for $f \in H^2$,
$$
\dx^2 R_0 f = R_0 \dx^2 f  -  \delta^{-2} f_{\vert_{x=\pm R}} e^{ - \frac{1}{\delta} | x|_R} . 
$$
Iterating these identities yields 
$$
\dx^{2l} R_0 f = R_0 \dx^{2l} f  -  \delta^{-2l}  \Big(  \sum_{ l' < l} ( \delta^2 \dx^2)^{l'}f_{\vert_{x=\pm R}} \Big)  
e^{ - \frac{1}{\delta} | x|_R} ,
$$
 $$
 \dx^{2l+1} R_0 f =  
  R_1 \dx^{2l+1}  f    
    \pm  \delta^{-2l-1}  \Big(  \sum_{ l' \le  l} (\delta^2 \dx^2)^{l'}f_{\vert_{x=\pm R}} \Big)  e^{ - \frac{1}{\delta} | x|_R}, 
 $$
proving the lemma. 
\end{proof}

Thus the lack of commutation of  $R_0$ with derivatives is encoded in the trace operators 
$\cD^\pm_k $. It is convenient and more symmetric to introduce 
$$
 \jump {\mathcal D_{k} f}  =   \mathcal D^+ _{k} f  -\mathcal D^- _{k} f, \qquad 
  \av{\mathcal D_k  f}  = \frac{1}{2} (    \mathcal D^+ _{k} f  +\mathcal D^- _{k} f ) . 
$$
The following corollary tells us that if $q$ is a function of the form
\begin{equation}\label{eqdefq}
q=R_0 f + \sigma  e^{ - \frac{1}{\delta} | x|_R}
\end{equation}
(as is $\cR (f,\rho)$ by \eqref{qind}), and if $q$ is bounded in $H^k(\cE)$ then necessarily the singularities of $R_0 f$ (explicited in Proposition \ref{lemtr1}) and of the exponential layer must compensate. Moreover, the second point of the corollary shows that the global contribution of these singularities can be controlled by $ \jump {\mathcal D_{k} f}$ and  $\av {\mathcal D_{k}f} -\sigma$.
\begin{corollary}
\label{cor41} 
For $k \ge 1$, there are  constants $C$  and  $C'$   such that for all $f \in H^k(\cE)$ and $\sigma \in \RR$ and 
$\delta \in (0, 1]$, 
the function $ q = R_0 f + \sigma  e^{ - \frac{1}{\delta} | x|_R}$ satisfies 
\begin{equation}
\label{tr2}
  |  {\mathcal D}_k^\pm f -\sigma |   \le    \delta^{k -1/2} C  (  \|   f  \|_{H^k(\cE_\pm)}  + \| q  \|_{H^k(\cE_\pm)}  )  
\end{equation}
and 
\begin{equation}
\label{tr1b} 
\| (q, \delta \dx q)    \|_{H^k(\cE)} \le C'     \|   f  \|_{H^k(\cE)}  +  
 C'  \delta^{ - k + 1/2} \big(  |  \jump {\mathcal D_{k} f}  + | \av{\mathcal D_k f}  - \sigma | \big).
\end{equation}
\end{corollary}

\begin{proof}
Note that
$$
\dx^l e^{  - \frac{1}{\delta} | x|_R}  =   (\mp 1)^l   \delta^{ -l} e^{  - \frac{1}{\delta} | x|_R} \qquad  \mbox{on }\quad \cE_\pm. 
$$
Hence, for  $0 < l \le k$, one has on $\cE_\pm$ 
$$
\dx^l q   = R_{\iota} \dx^l f   +  (\mp \delta)^{-l} \Big( 
\sigma - \mathcal D_l^\pm   (f )  \  \Big) e^{  - \frac{1}{\delta} | x|_R} . 
$$
Thus 
$$
 \delta^{ - l+1/2}  |\sigma  -  \mathcal D^\pm_k (f) | \le      \| \dx^k q \|_{L^2(\cE_\pm)} + \|  \dx^k f \|_{L^2(\cE_\pm)} 
$$
since   $\| R_\iota \|_{ L^2 \mapsto L^2} \le 1 $. 
Taking $l = k$, and noticing that 
$$
|  \mathcal D_{k}^+ f - \sigma |  + |  \mathcal D_{k}^-  f - \sigma |   \approx 
| \jump{ \mathcal D_{k}^\pm f }  | + |  \av{\mathcal D_{k}f }- \sigma | , 
$$ 
 \eqref{tr2} follows.

Moreover,    $\| \delta \dx R_\iota \|_{ L^2 \mapsto L^2} \le 1$. Therefore, for $l \le k$ 
$$
\| (\dx^l q , \delta \dx^{l+1} q)    \||_{L^2(\cE_\pm)} 
 \le     \|   f  \|_{H^k(\cE_\pm)}  +  \delta^{ - l +1/2} |  \sigma  - \mathcal D_l^\pm   (f )  | . 
$$
Note that, for $l < k$, one has 
$$
| \mathcal D^\pm_l (f) - \mathcal D^\pm_k (f) | =  \Big| \sum_{l \le 2l' < k} (\delta \dx)^{2l'} f _{| x = \pm R} \Big| \le
C \delta^{l} \| f \|_{H^k} 
$$
so that 
$$
\delta^{ - l+1/2} \big|\sigma  -  \mathcal D^\pm_l (f)  \big| \le    \delta^{ - l+1/2}  \big|\sigma  -  \mathcal D^\pm_k (f) \big|  +  
C  \delta^{1/2 } \| f \|_{H^k} . 
$$
Together with \eqref{est25} when $l= 0$, this implies  that 
$$
\| (q, \delta \dx q)      \|_{H^k(\cE_\pm)} \le C   \Big(  \|   f  \|_{H^k(\cE_\pm)}  +  
\delta^{ - k + 1/2} |  \mathcal D_{k}^\pm f - \sigma | \Big) . 
$$
and the  estimate \eqref{tr1b} follows. 
\end{proof} 

\subsubsection{Taylor expansions of $R_0$ and $R_1$}

In \eqref{qind}, the quantity $\cR(f,\rho)$ is defined through an equation of the form \eqref{eqdefq} but with a scalar  $\sigma $ depending  itself on $\jump{ \delta^2 \dx R_0 f }$. We are therefore led to study the Taylor expansion of terms  of the form $R_0 f$ and, through Proposition \ref{lemtr1}, of $R_1f$. We start with  a useful preliminary  result. Introduce the non local trace operators 
$$
 I^\pm(f)  = \delta^{-1} \int_{\cE_\pm} e^{-  \frac{1}{\delta} |y |_R}  f(y) {\rm d}y, 
$$
from $L^2(\cE_\pm) $ to $\RR$, which satisfy
$$
 | I^\pm   (f) |  \le  (2\delta)^{- 1/2}  \| f \|_{L^2 (\cE_\pm)}. 
$$

\begin{proposition}\label{nonlocal}
For $f  \in L^2 (\mathcal E_\pm)$,   
$$ R_1 f _{| x = \pm R}  = I^\pm (f)$$
and for   $f  \in H^k (\mathcal E)$ with $k \ge  1$, 
\begin{equation}
\label{tr8} 
I^\pm (f) =   \sum_{l=0}^{k-1}    (\pm \delta \dx)^l  f_{| x = \pm R}   + (\pm \delta )^k  I^\pm (\dx^k f) . 
\end{equation}
\end{proposition}

\begin{proof}
For 
$f \in H^1$, 
since $ e^{-  \frac{1}{\delta} |y |_R}  = \mp   \delta \dx ( e^{-  \frac{1}{\delta} |y|_R} )$ one has 
$$
    I^\pm (f)    =       f_{| x = \pm R}  \   \pm  \ I^\pm  ( \delta \dx f) . 
$$
Iterating this identity yields  \eqref{tr8}.  In particular, if     $u$ satisfies  $u - \delta^2 \dx^2 u = f$, then 
$$
I^\pm (f) = I^\pm (u) - I^\pm (\delta ^2 \dx^2 u) =  u _{| x =\pm R} \pm \delta \dx u_{| x =\pm R}. 
$$
Thus if  $u = R_1 f$, then $u_{| x =\pm R} = I^\pm (f)$. 
\end{proof}
Together with Proposition \ref{lemtr1}, this result can be used to obtain Taylor expansions of $R_0 f$ at $x=\pm$ in terms of $\delta$.
\begin{corollary} \label{deriveop}
Suppose that $f \in H^k(\cE)$, with $k>0$.  Then, for $0\leq l \leq k$, 
\begin{equation}
\label{tr5}
(\dx^l R_0 f) _{\vert  x = \pm R}  =    \delta^{-l} {\mathcal R}^\pm_{l, k}   f   \pm  \iota  (\pm \delta)^{k-l} 
I^\pm (\dx^k f) 
\end{equation}
where  
\begin{itemize}
\item when $l$ is even then $\iota = 0$ and 
$$  {\mathcal R}^\pm_{l, k}  f =  -  \cD^\pm_l f $$
\item when $l$ is odd then $\iota = 1$ and 
$$
  {\mathcal R}^\pm_{l,k}   f  =  \pm  \Big(   {\mathcal D}^\pm_l  f  \ +  \ 
 \sum_{l'=l }^{k- 1}    (\pm \delta)^{l'}  \dx^{l'}  f_{| x = \pm R} \Big).
$$
\end{itemize}
\end{corollary}

\begin{proof}
The identity \eqref{tr5} follows immediately from Proposition~\ref{lemtr1} when $l$ is even. When $l$  is odd, 
one has
$$
\dx^l R_0  f_{| x = \pm R}  = \pm \delta^{ -l}    {\mathcal D}^\pm_{l}   f   +  R_1 (\dx^l f)_{| x = \pm R}. 
$$
Since $l$ is odd, $\pm \delta^l = (\pm \delta)^l$, and  Proposition \ref{nonlocal} implies that 
$$
\pm \delta^l  R_1 (\dx^l f)_{| x = \pm R}  = \sum_{l'=0 }^{k- l - 1}    (\pm \delta)^{l' +l}  \dx^{l+ l'}  f_{| x = \pm R}   +
 (\pm \delta )^{k}   I^\pm (\dx^k f) . 
$$
  and \eqref{tr5} follows. 
\end{proof} 

\subsubsection{Compatibility conditions for the control of $\cR(f,\rho)$}
 From Corollary \ref{deriveop} we get in particular, 
\begin{equation*}
   {\mathcal R}^\pm_{1,k}   f  =   \pm \big(    f^\pm _{| x = \pm R}   +  
 \sum_{l =1 }^{k- 1}    (\pm \delta)^{l}  \dx^{l}  f_{| x = \pm R} )
  =  \pm \cS^\pm_k f , 
\end{equation*}
where
$$
\cS_k^\pm f = \sum_{l=0 }^{k- 1}    (\pm \delta)^{l'}  \dx^{l}  f_{| x = \pm R} . 
$$
Thus 
$$
\delta^2 \dx R_0 f _{| x = \pm R}  =  \pm  \delta  {\mathcal S}^\pm_{ k} f 
+  (\pm \delta )^{k} \delta  I^\pm (\dx^k f) .
$$ 
so that 
\begin{equation}
\label{est219}
\Big| \jump{ \delta^2 \dx R_0 f }  -   2  \delta \av{ {\mathcal S} _ k f}  \Big| \le  
2 \delta^{k+1/2}  \| \dx^k f \|_{L^2}   . 
\end{equation}
 with $\av{ {\mathcal S} _ k f} = \frac{1}{2} ( {\mathcal S}^+ _ k f+  {\mathcal S}^- _ k f)$. A uniform control of $q=\cR(f,\rho)$ can then be reduced to a control on the quantities $A$ and $B$ defined in the statement below.
\begin{proposition} 
\label{propb1}
For $k \ge 1$, there is a constant $C$ such that for all $f \in H^k$ and  $\rho \in \RR$ 
the function  $q=\cR(f,\rho)$ with $\cR$ defined by \eqref{qind} satisfies 
\begin{equation}
\label{estAB} 
| A | + |B | \le C \delta^{ k - 1/2} \big( \| f \|_{H^k(\cE)} + \| q \|_{H^k(\cE)} \big) 
\end{equation}
and 
\begin{equation}
\label{estq} 
\| (q , \delta \dx q )  \|_{H^k(\cE)} \le C     \|   f  \|_{H^k(\cE)}  +  
 C \delta^{ - k + 1/2} \big(  |  A  | 
 + |  B |  \big), 
\end{equation}
where
$$
A =  \jump {\mathcal D_{k} f} , \qquad 
 B =  \alpha  \av{\mathcal D _{k} f}   -  2 \delta \av{\mathcal P_k  f}  - \rho 
$$
and  
$$
\mathcal P^\pm_k f  = \sum_{2l+1 < k} (\pm \delta \dx)^{2l+1} f _{| x = \pm R} . 
$$
\end{proposition} 

\begin{proof}
Apply  Corollary~\ref{cor41}    to $f$ and $\sigma$ with
$$
\sigma = \frac{1}{\alpha + 2 \delta } \big(\rho + \delta^2 \jump{ \dx R_0 f } \big).   
$$
Using  \eqref{est219} and noticing that $    {\cS}^{\pm}_ k f - \cD^{\pm}_k f  = 
\cP^{\pm}_k f  $, one has  
$$
(\alpha + 2 \delta) (\av{\cD_k f} - \sigma ) 
= \alpha  \av{\cD_k f}  -  2 \delta  \av{\cP_k f}  - \rho  + O( \delta^{k+1/2}) \| f \|_{H^k}  
$$
and 
\begin{equation}
\label{appsigma} 
\Big| (\alpha + 2 \delta) (\av{\cD_k f} - \sigma )  -  B \Big| \le C  \delta^{k+1/2} \| f \|_{H^k} . 
\end{equation}
Hence \eqref{estAB} and \eqref{estq} follow from \eqref{tr2} and \eqref{tr1b} respectively. 
\end{proof} 

\begin{remark}
\textup{An important fact is that  the operators $\mathcal D,   \mathcal S,   \mathcal P$  are polynomial functions of $\delta$, and thus smooth up to $\delta = 0$. In particular, at  $\delta = 0$, they simplify to   } 
$$
\mathcal D^\pm _k f =   \mathcal S^\pm _k f= f_{| x = \pm R}   . 
$$
\end{remark}
In addition to the estimates  of the proposition above, one has a precise  description of the Taylor expansion of 
$q$   at  $x = \pm R$.

\begin{proposition}
\label{propTaylor}
For $k \ge 1$, there is a constant $C$ such that for all $f \in H^k$ and  $\rho \in \RR$ 
the Taylor expansion at $x = \pm R$  of $q=\cR(f,\rho)$ with $\cR$ defined by \eqref{qind} satisfies  for $l < k$
\begin{equation}
\label{taylor}
\begin{aligned}
\Big|   \dx^l q_{\,| x = \pm R} &  -    {\mathcal D}^\pm_{k -l}  \dx^l f  \Big|   
\\ 
& \le C  \delta^{k- l - 1/2} \left(    \|   f  \|_{H^k(\cE)}  +  
   \delta^{ -  k + 1/2 } \big( | A|  + | B| \big) \right), 
\end{aligned}
\end{equation} 
\end{proposition}

\begin{proof} Using \eqref{eqdefq} with $\sigma = \big(\rho + \delta^2 \llbracket\partial_x R_0 f\rrbracket\big)/(\alpha+2\delta)$, we get  by using \eqref{tr5} that
\begin{equation*}
\Big| \delta ^l \dx^l q_{\,| x = \pm R}  -     \mathcal {R}^\pm_{l, k} f 
 - (\mp 1)^l \sigma   \Big|  \le C \delta^{k - 1/2}  \| f\|_{H^k}. 
\end{equation*}
With \eqref{appsigma} one can replace $\sigma$ by $   \av{\mathcal {D}_{ k} f} $ up to an 
error of size $|B|$, and  next  by $ { \mathcal D}^\pm_{k} f $ up to additional error of size 
$ |A |$. Hence
\begin{equation*}
\Big| \delta ^l \dx^l q_{\,| x = \pm R}  -     \mathcal {R}^\pm_{l, k} f 
 - (\mp 1)^l \cD^\pm_k f  \Big|  \le C  \big( |A | + | B |  +   \delta^{k - 1/2}   \| f\|_{H^k} \big) 
\end{equation*}
Now we use Corollary \ref{deriveop}. When $l$ is even,  $\cR^\pm_{l, k} =  - \cD^\pm_l $  and 
$$
  \mathcal D^\pm_k f +    \mathcal R^\pm_{l,k}  f  =      \sum_{l \le 2l' < k} (\delta \dx)^{2l'} f _{| x = \pm R}
  =  \delta^l  \cD^\pm_{k-l} \dx^l f . 
$$
When $l$ is odd, 
$ {\mathcal R}^\pm_{l,k}   f  =  \pm  \Big(   {\mathcal D}^\pm_l  f  \ +  \ 
 \sum_{l'=l }^{k- 1}    (\pm \delta)^{l'}  \dx^{l'}  f_{| x = \pm R} \Big) $ and 
 $$
 \begin{aligned}
 \mp  \cD^\pm_k f +  {\mathcal R}^\pm_{l,k}   f  & = \mp \sum_{l < 2l' < k } (\pm\delta \dx)^{2l'}  f_{| x = \pm R}  
 \pm  \sum_{ l \le l' < k }    (\pm \delta)^{l'}  \dx^{l'}  f_{| x = \pm R} 
 \\
 & = \pm  \sum_{l \le 2l' +1 < k}     (\pm \delta)^{l'}  \dx^{2l'+1}  f_{| x = \pm R}    =  \delta^l  \cD^\pm_{k-l} \dx^l f . 
 \end{aligned}
 $$
In both case, this proves \eqref{taylor}. 
\end{proof}
 
According to Proposition \ref{propTaylor}, one can approximate $\av{{\mathcal D}_k f}\approx \av{q}$ and  $2\delta \av{{\mathcal P}_k f} \approx \delta^2 \jump{\dx q}$. Together with Proposition \ref{propb1}, this implies the following Corollary that motivates the compatibility conditions \eqref{cc1new}.
\begin{corollary} 
\label{corob1}
For $k \ge 1$, there is a constant $C$ such that for all $f \in H^k$ and  $\rho \in \RR$ 
the function  $q=\cR(f,\rho)$ with $\cR$ defined by \eqref{qind} satisfies 
\begin{equation}
\label{estABtilde} 
| \widetilde{A} | + |\widetilde{B} | \le C \delta^{ k - 1/2} \big( \| f \|_{H^k(\cE)} + \| q \|_{H^k(\cE)} \big) 
\end{equation}
and 
\begin{equation}
\label{estqtilde} 
\| (q , \delta \dx q )  \|_{H^k(\cE)} \le C     \|   f  \|_{H^k(\cE)}  +  
 C \delta^{ - k + 1/2} \big(  |  \widetilde{A}  | 
 + |  \widetilde{B} |  \big), 
\end{equation}
where
$$
\widetilde{A} =  \jump {q} , \qquad 
 \widetilde{B} =  \alpha  \av{q}   -  2 \delta^2 \jump{\dx q}  - \rho .
$$
\end{corollary}
 
 \subsection{Proof of Proposition \ref{estID}}
  We show here the uniform bounds of Proposition \ref{estID}. Consider initial data  $(\theta_0^{in}, q_0^{in}) \in H^{n+1} (\mathcal E)\times H^{n+2}(\mathcal E)$ and parameters 
$\delta \in (0,1)$ and  $\eps \in [0, 1]$  satisfying 
$$
\jump{q_0^{\rm in}}=0, \quad \| \theta_0^{\rm in} \|_{H^{n+1} (\mathcal E)}  \le M ,
 \quad \|  ( q_0^{\rm in}, \delta \dx q_0^{\rm in} )  \|_{H^{n+1} (\mathcal E)}  \le M, 
$$
$$
 1+\eps c'(\theta_0^{\rm in }) \geq M  . 
$$
for some given $M>0$.  Let us also introduce the notation
 $$
 U^\pm_{j, k}  = \dx^k U^{\rm in}_j{}_{\vert x = \pm R} 
 $$
 which are defined through the induction relation \eqref{inducid}. We also assume that the compatibility conditions \eqref{cc1new} hold, that is, 
  \begin{equation}
\label{cc1newnew} 
\begin{cases}
\big\lvert \jump{q_{j+1}^{\rm in}} \big\rvert  \le M \delta^{n-j-1/2} \\
\big\lvert  \alpha \av{q_{j+1}^{\rm in}} +\jump{\theta_{j}^{\rm in}} -\delta^2 \jump{\dx q_{j+1}^{\rm in}} \big\rvert   \le M \delta^{n-j-1/2} 
\end{cases}
 \qquad  \mathrm{for} \ \  0 \le j \le n -1.
\end{equation}
We will prove by induction on $j \le n+1 $ that  there are constants
$M_j$ which depend only on $M_0=M$  such that
\begin{align} 
\label{inducjnew}
 & \|  ( \theta_j^{\rm in}, q_j^{\rm in}, \delta \dx q_j^{\rm in} )  \|_{H^{n+1 - j } (\mathcal E)}  \le M _j, 
\end{align}
agreeing that the condition \eqref{inducjnew} is void at the final step  $j = n +1$.

When $j = 0$, \eqref{inducjnew} is our assumption. We  assume that the estimates \eqref{inducj} and \eqref{inducjk} are satisfied up to the order 
$j < n $ and prove them at the order $j+1$. \\
The multiplicative properties of Sobolev spaces show  that $\theta_{j+1}^{\rm in} = - \Phi_j^{\rm in}$ satisfies \eqref{inducjnew} for some 
$M_{j +1} = C(M_0, \ldots, M_j)$. The fact that $q_{j+1}^{\rm in}$ also satisfies \eqref{inducjnew} is a direct consequence of Corollary \ref{corob1}. 
 This proves that \eqref{inducjnew} is satisfied for all  $j \le n$.  For the case $j=n+1$, we first get as above that 
$$
\big\| \Phi_n^{\rm in} , \Gamma_n^{\rm in} \|_{L^2 (\cE)} \le C(M_0, \ldots, M_j) . 
$$
This give a bound for the $L^2$ norm of  $\theta_{n+1} = - \Phi_n$. Moreover, since $R_0$
and $\delta^2 \dx R _0$ are uniformly bounded from $L^2 $ to $L^2$ and $H^1$ respectively, one has 
$$
\| q_{n+1}^{\rm in} \|_{L^2} \le  C \big(  \| \Gamma_n^{\rm in} \|_{L^2} + \delta^{1/2} \| \theta_{n}^{\rm in} \|_{H^1} \big),
$$
which  proves that \eqref{inducjnew} is also satisfied for all  $j =  n+1$, and the proof of Proposition~\ref{estID}
 is complete.

\subsection{Approximate compatibility conditions}

One issue with the compatibility conditions \eqref{cc1new} is that they are nonlocal and very difficult to check. This is the reason why we show here that it is possible to replace them by the simpler compatibility conditions \eqref{cc1newapp} that only involve the Taylor expansion of the initial condition at the boundaries $x=\pm R$. Moreover, as we shall see, as $\delta \to 0$ these new converge to the usual  compatibility conditions of the hyperbolic boundary value problem. To explain that,   we first recall briefly the analysis in  the hyperbolic case.   
\subsubsection{Hyperbolic compatibility conditions}

 When $\delta = 0$ (hyperbolic case), the induction is 
\begin{equation}
\label{inducid0} 
\theta_{j+1}^{\rm in} =  - \Phi^{\rm in}_{j}  , \qquad 
 q_{j+1}^{\rm in}  = - \Gamma _j^{\rm in} .  
\end{equation} 
In this case,   if   
$U^{\rm in}_0 \in H^{n+1} (\mathcal E)\times H^{n+1}(\mathcal E)$, the $U^{\rm in}_j $ are defined in $H^{n+1-j} (\cE) \times H^{n+1-j} (\cE)$ only for $j \le n+1$. 
The \emph{compatibility conditions} state that the boundary conditions are satisfied in the sense of Taylor expansion
in time. They read $\jump{ q^{\rm in}_0 } = 0$ and 
\begin{equation}
\label{cchyp}
\jump { q_{j+1}^{\rm in}  } = 0, \qquad  \alpha  q_{j+1}^{\rm in , i }  + \jump{\theta_j^{\rm in} } = 0, \qquad 
\mathrm{for} \ \ 
0 \le j \le n -1. 
\end{equation}
   
The $j$-condition  \eqref{cchyp} bears  only on the Taylor expansions at order $j+1$  of $U^{\rm in}_0$ at $x = \pm R$. 
More precisely,  for $k \le n - j$,  let  
 $U^ {\pm}_{j, k} = \dx^k  U^{\rm in} _j {}_{\vert x = \pm R}   $ 
 denote the coefficients of the Taylor expansion of $U^{\rm in}_j  \in H^{n+1 -j}$. Then, for $k < n-j$, 
 \begin{equation}
 \label{inducTaylor1}
\left\{ \begin{aligned}
 &  \Phi^\pm _{j, k} =  \dx^k  \Phi^{\rm in} _j {}_{\vert x = \pm R}   = 
 \cT_{j,k}  \big( \{ U^\pm_{j', k'} \}_{j' \le j , k' \le k+1} \big) , 
 \\ 
& \Gamma^\pm _{j, k} =  \dx^k  \Gamma^{\rm in} _j {}_{\vert x = \pm R}   
=  \cG_{j, k} \big( \{ U^\pm_{j', k'} \}_{j' \le j , k' \le k+1} \big) , 
\end{aligned}\right. 
\end{equation} 
where $\cT_{j,k}$ and $\cG_{j,k}$
are polynomials of   $(1 + \eps c' (\theta_{0,0}))^{-1} $  and of the 
 $ U^\pm_{j', k'} $ for  $j' \le j $ and $ k' \le k+1 $.  By \eqref{inducid0} 
\begin{equation}
\label{inducTaylor2}
 U^ {\pm}_{j+ 1, k}  = - \begin{pmatrix}  \cT_{j,k}  \big( \{ U^\pm_{j', k'} \}_{j' \le j , k' \le k+1} \big) 
 \\ \cG_{j,k}  \big( \{ U^\pm_{j', k'} \}_{j' \le j , k' \le k+1} \big) \end{pmatrix}
\end{equation}
 with leading term
 $$
 U^\pm_{j+1, k} =  \A ^\pm  \ U^\pm_{j , k+1}   + \mathcal U^\pm _{j, k}  , 
\qquad \A^\pm = \begin{pmatrix}  0 & - \frac{1}{1 + \eps c' (\theta^\pm _{0, 0}) } 
\\
 -1  & -2 \eps q_{0,0}  \end{pmatrix} 
 $$
 where $\mathcal U^\pm_{j,k } (\{ U^\pm_{j', k'} \}_{j' \le j , k' \le k} )  $ is a smooth function of the coefficients   $U^\pm_{j', k'} $  
 with $j' \le j $ and $k' \le k$. 
Hence by induction, this implies that for $0 \le j < n$, 
\begin{equation}
\label{coinhyp}
U^\pm_{j+1, 0} = ( \A^\pm)^{j+1}     U^\pm_{0, j+1}  + \widetilde{ \mathcal  U} _{j}  (U^\pm_0, \ldots,  U^\pm_{0,j}  )  , 
\end{equation}
  where $\widetilde { \mathcal U _{j} } $ is a smooth function of its arguments $ (U^\pm_0, \ldots,  U^\pm_{0,j}  )$.  
  Hence the compatibility conditions \eqref{cchyp} which read
  \begin{equation}
  \label{cchyp2}
  \jump{ q_{j+1, 0} } = 0 , \qquad  \alpha \av{q_{j+1, 0} }+   \jump{\theta_{j, 0} } = 0 
  \end{equation}
are equations for $\{ U^\pm _{0, k} \} _{k \le j+1} $ which, thanks to \eqref{coinhyp}, can be explicitly solved 
by induction on $j$.  

\subsubsection{The dispersive case}

 When $\delta > 0$ (dispersive case), the situation is quite different  since  $q_{j+1} $ is given by a nonlocal operator, so that the 
Taylor expansion of $q_{j+1}$ is not \emph{exactly} a function of the sole Taylor expansions of the previous terms. 
However, \emph{when the approximate compatibility conditions \eqref{cc1newapp} are satisfied up to order $j$}, this remains 
true \emph{approximately}, so that the $(j+1)$-th condition is equivalent to a condition on the Taylor expansions. 
Indeed, Proposition \ref{propTaylor} tells us that it is then possible to make the approximation
$$
 q^\pm_{j+1, k}  \sim   -  \sum_{l <  (n - k - j ) / 2 } \delta^{2l} \Gamma^\pm_{j, k + 2l} .
 $$ 
If we replace the nonlocal formula $q^\pm_{j+1,k}=- \dx^k {\mathcal R}(\Gamma_j^{\rm in},\jump{\theta_j^{\rm in}})$ of \eqref{inducid} by this approximation, we are led to define a sequence $\widehat{U}_{j,k}^{\rm in}$ of approximations of $\dx^k U_j^{\rm in}$ as follows:
$$
\widehat U^\pm_{0, k}  = U^\pm_{0, k} , \qquad k \le n,
$$
and we define
 \begin{equation}
 \label{inducTaylor3}
\left\{ \begin{aligned}
 & \widehat{ \Phi}^\pm _{j, k}   = 
 \cT_{j,k}  \big( \{\widehat U^\pm_{j', k'} \}_{j' \le j , k' \le k+1} \big)   
 \\ 
&\widehat{ \Gamma} ^\pm _{j, k} =    
  \cG_{j, k} \big( \{ \widehat U^\pm_{j', k'} \}_{j' \le j , k' \le k+1} \big)   
\end{aligned}\right.  \ , \qquad   k + j < n , 
\end{equation}
where the operators $\cT_{j,k}$ and $ \cG_{j, k}$ are defined as in \eqref{inducTaylor1} in the hyperbolic case, and next 
$$
\left\{ \begin{aligned}
 &  \widehat \theta^\pm_{j+1, k}    =  -    \widehat{ \Phi}^\pm _{j, k} 
 \\ 
& \widehat q^\pm _{j+1, k} =  - \sum_{0 \le l < (n - k -j )/2} 
 \delta^{2l} \widehat \Gamma^\pm _{j, k+2 l}  ,  
\end{aligned}\right. 
$$
thus defining the $\widehat U^\pm_{j, k}$ for $j + k \le n$. 
Note that are polynomials of 
  $(1 + \eps c' (\theta_{0,0}))^{-1} $, the coefficients $ \{ U^\pm_{0, k}\}_{k \le n}$ and $\delta$ so that they are well defined if $\delta=0$ in which case they coincide with the relation \eqref{inducTaylor2} obtained in the hyperbolic case. \\
 We can consequently derive an approximation of the compatibility conditions \eqref{cc1newnew}, namely, 
   that there exists $M>0$ such that
 \begin{equation}
\label{cc1newnewapp} 
\begin{cases}
\big\lvert \jump{\widehat{q}_{j+1,0}} \big\rvert  \le M \delta^{n-j-1/2} \\
\big\lvert  \alpha \av{\widehat{q}_{j+1,0}} +\jump{\widehat{\theta}_{j,0}} -\delta^2 \jump{\widehat{q}_{j+1,1}} \big\rvert   \le M \delta^{n-j-1/2} 
\end{cases}
 \qquad  \mathrm{for} \ \  0 \le j \le n -1.
\end{equation}
We now show that the results of Proposition \ref{estID} remain true under these approximate compatibility conditions.
\begin{proposition} 
\label{estIDfin}
Given $M$, there is a constant $C$ such that  
for all initial data $(\theta_0^{\rm in}, q_0^{\rm in}) \in \HH^n$ and parameters 
$\delta $ and  $\eps$ in $[0, 1]$  satisfying \eqref{H0} and \eqref{H1}, and the approximate compatibility  conditions \eqref{cc1newnewapp} for all $j < n$,
 one has  
$$
 \big\| ( \theta_j^{\rm in},  q_j^{\rm in}, \delta \dx q_j^{\rm in})   \big\|_{H^{ n+1- j} (\mathcal E)} \le C 
 \qquad for \ \  0 \le j \le n +1. 
$$
\end{proposition} 

\begin{remark}
\textup{The conditions   \eqref{cc1newnewapp} are $2 n $ explicit inequalities for the unknowns $U^\pm_{0, k}$, which supplement 
the original jump condition  $\jump{q_{0,0}}  = 0$ assumed in \eqref{H0}. The 
functions  $\av{\widehat{q}_{j+1, 0}} $,   $ \jump { \widehat \theta_{j, 0} }$ and $\jump{\widehat{q}_{j+1,1}}$  are smooth functions of the $U^\pm_{0, k}$ and $\delta$. 
 When   $\delta = 0$ ,   one has $\widehat q_{j, k} = q_{j, k} $ and
 the conditions \eqref{cc1newnewapp}  reduce to \eqref{cchyp2}. In particular, for $\delta $ small, the 
 set of initial data satisfying \eqref{cc1newnewapp}
is a smooth variety.} 
 \end{remark} 
 
\begin{remark}
\textup{If the initial data are supported away from the boundary $x = \pm R$, or more generally if their Taylor expansion at order 
$n$ vanish at $x=\pm R$, then the conditions \eqref{cc1newnewapp} are satisfied. } 
\end{remark}


\begin{proof}

%
Following the same path as in the proof of Proposition \ref{estID} but controlling in addition the size of the error $U_{j,k}^\pm - \widehat{U}_{j,k}$, we will prove by induction on $j \le n+1 $ that  there are constants
$M_j$ which depend only on $M=M_0$  such that
\begin{align} 
\label{inducj}
 & \|  ( \theta_j^{\rm in}, q_j^{\rm in}, \delta \dx q_j^{\rm in} )  \|_{H^{n+1 - j } (\mathcal E)}  \le M _j, 
\\
\label{inducjk}
& \big| U^\pm_{j, k} -  \widehat U^\pm_{j, k} \big|  \le M_j \delta^{n-j-k+1/2} \qquad  \mathrm{ for} \  k < n-j, 
\end{align}
agreeing that the condition \eqref{inducjk} is void at the final step  $j = n +1$. \\
When $j = 0$, \eqref{inducj} is our assumption, and \eqref{inducjk} is trivial since 
$\widehat U^\pm_{0, k} = U^\pm_{0, k}$. \\
We now assume that the estimates \eqref{inducj} and \eqref{inducjk} are satisfied up to the order 
$j < n $ and prove them at the order $j+1$. 

The multiplicative properties of Sobolev spaces immediately imply the following  estimates 
$$
\big\| \Phi_j , \Gamma_j \|_{H^{n-j} (\cE)} \le C(M_0, \ldots, M_j) 
$$
where $C(M_0, \ldots, M_j) $ is  constant which depends only on $ (M_0, \ldots, M_j) $. 
In particular, this shows that $\theta_{j+1} = - \Phi_j$ satisfies \eqref{inducj} for some 
$M_{j +1} = C(M_0, \ldots, M_j)$. 
 
With notations as in \eqref{inducTaylor1}, the  Taylor expansions at $x = \pm R$ of $ \Phi_j $ and $ \Gamma_j $  are  given by functions 
$\cT_{j, k} $ and $\cG_{j, k} $ which are polynomials of  $(1 + \eps c (\theta^{\rm in}_0))^{-1}$ and 
 $U^{\pm}_{j', k'} $ with $j' \le j$ and $j' +k' \le j+k +1$. 
The induction hypothesis \eqref{inducj} implies that the  coefficients  $U^{\pm}_{j', k'} $ 
for $j' \le j$ and $j + k \le n $ remain in a ball of radius  $ C(M_0, \ldots, M_j) $, on which the derivatives of the functions
$\cT_{j, k} $ and $\cG_{j, k} $ are bounded. Hence comparing \eqref{inducTaylor1} and
 \eqref{inducTaylor3} implies that
\begin{equation*}
\begin{aligned}
\big| \Phi_{j, k} & - \widehat \Phi_{j, k} \big|  + 
\big| \Gamma_{j, k} - \widehat \Gamma_{j, k} \big| 
 \\ \le  &  C( M_0, \ldots, M_j) 
   \sup_{j' \le j,   k' \le     k+1}    \big| U^\pm_{j', k'} -  \widehat U^\pm_{j', k'} \big| 
\end{aligned}
\end{equation*}
and using the induction hypothesis \eqref{inducjk}, one obtains  that for $j+1 + k < n$, 
$$
\big| \Phi_{j, k} - \widehat \Phi_{j, k} \big|  + 
\big| \Gamma_{j, k} - \widehat \Gamma_{j, k} \big| 
\le  C( M_0, \ldots, M_j)  \delta^{n-j-k-1/2}   . 
$$
In particular, this shows that $  \theta_{j+1, k} - \widehat \theta_{j+1, k} $ satisfies  \eqref{inducjk}. By Proposition \ref{propTaylor} and \eqref{cc1newnewapp} we get that $  q_{j+1, k} - \widehat q_{j+1, k} $ also satisfies \eqref{inducjk} and that a similar upper bounds also holds for $\delta^2 \dx q_{j+1}^\pm -\delta^2 \widehat{q}^\pm_{j+1,1}$. The compatibility condition \eqref{cc1newnew} is therefore a direct consequence of its approximate version \eqref{cc1newnewapp}.
%
%
Hence, we can use Proposition~\ref{propb1} to obtain  \eqref{inducj}. By induction, we have then proved that \eqref{inducj} holds for $j\leq n$. For $j=n+1$, one gets the result as in the proof of Proposition \ref{estID}.
%
 \end{proof}
 
 
 \subsection{An additional bound}
 
 By Proposition \ref{estIDfin}, we have a control of $U_j^{\rm in} $ in $H^{n+ 1 -j}(\cE)$, and thus of the boundary values
 $U ^\pm_{j, k}$ for $k + j \le n$.  The proof of the main theorem uses an $L^2$ energy estimate  of $U_j$
 which requires and additional  control of $  \av{q_{j}^{\rm in}} = \frac{1}{2} ( q^+_{j, 0} + q^-_{j, 0}) $. 
 When $j \le n$  if follows form the estimate above. However, we need to control the time derivatives  \emph{up to $ j = n +1$},  and we now provide the needed control of   $  \av{q^{\rm in}_{n+1} }$.

By Proposition~\ref{prop1}, we know that for initial data  
$U_0^{\rm in} $ in $\HH^n$ satisfying
$\jump{ q_0^{\rm in}}  = 0$ and $  \inf \{ 1 + \eps c' (\theta_0^{\rm in} \}  > 0$, there is a unique solution 
$C \in C^1 ([0, T^*), \HH^n)$ for some $T ^* > 0$. 
Using the o.d.e formulation  \eqref{nlode}, 
we see that $U$ is indeed  $ C^\infty ([0, T^*), \HH^n)$, so that the 
$U_j $ are defined for all $j$ and belong to  $\HH^n$. So the quantities 
 $ \av{ q_j^{\rm in}} $ are well defined. 
 Moreover, by 
 \eqref{dtqi} we know that on $[0, T^*)$, 
 \begin{equation*}
\label{dtqi4}
\frac{d}{dt} \av{q}  = -  \frac{1}{\alpha + 2 \delta} \Big( \delta^2 \jump{ \dx R_0 \Gamma  }  + \jump{ \theta } \Big). 
\end{equation*}
 Hence, for all $j $, 
 \begin{equation*}
 \av{q_{j + 1}} =  -  \frac{1}{\alpha + 2 \delta} \Big( \delta^2 \jump{ \dx R_0 \Gamma_j  }  + \jump{ \theta_j } \Big).
 \end{equation*} 
 In particular, 
 \begin{equation*}
 \big|  \av{q^{\rm in}_{j + 1}} \big|  \le C \big(\|  \Gamma_j^{\rm in}  \|_{L^2(\cE)}    + \|  \theta_j^{\rm in} \|_{H^1(\cE)} \big) . 
 \end{equation*} 
 Under the assumptions of Proposition~\ref{estIDfin}, we have a uniform control of the $U_j^{\rm in} $ 
 in $H^{1}$ for $j \le n$, and hence of the $\Gamma_j$ in $L^2$. Hence we have the following
additional estimate. 
\begin{proposition}
\label{addc} 
Under the assumptions of Proposition~\ref{estIDfin}, there is a constant $C$ such that for all 
$   j \le n+1$ ,
\begin{equation*}
 \big|  \av{q^{\rm in}_{j }} \big|  \le C. 
\end{equation*} 

\end{proposition}



\section{Uniform estimates of time derivatives}
\label{estimates}

The goal of  this section is to prove the  \emph{a priori} estimates for the $\dt^j U$ which are stated in Proposition~\ref{Estdtj} below. From the beginning, we know that there is a conserved energy for the nonlinear system. 
The system obtained for time derivatives is a linearized  version of of \eqref{Boussthetanl}--\eqref{Transs2},
 and we first prove $L^2$ energy estimates for it. Controlling the commutators, they will imply estimates for the 
 time derivatives of the solution.

 \begin{notation}
  Given "constants" $C_0, m, M, \dots$, we denote by 
$C( C_0, M, m, \dots)$ or $\gamma (C_0, M, m ,\dots)$ a function of these constants, which may change from one line to another. It is assumed that the dependence of these functions on their arguments is nondecreasing and that they  are independent of $\eps$ and $\delta $ in $[0, 1]$.
 \end{notation}

\subsection{$L^2$-bounds}

Consider the following  linearized version of \eqref{Boussthetanl}--\eqref{Transs2} around some reference state $(\underline{\theta},\underline{q})$, with non-homogeneous source terms
\begin{equation}\label{Bousstheta_nh}
\begin{cases}
(1+\eps c'(\utheta))\dt \theta+ \dx q= \varepsilon f,\\
[1- \delta^2 \dx^2] \dt q+2\varepsilon \uq\dx q +\dx \theta= \varepsilon g
\end{cases}
\quad \mbox{ on }\quad {\mathcal E},
\end{equation}
with transmission conditions
\begin{align}
\label{CBl1_nh}
\jump{q}&=0,\\
\label{CBl2_nh}
-\delta^2\dt \jump{\dx q}+\jump{\theta}&=-\alpha \dt \av{q}.
\end{align}
We derive here an a priori bound for the total energy associated to this linear system,
$$
E^{\rm tot}_\uU(U,\av{q})=E^{\rm ext}_\uU(U)+\frac{1}{2}\alpha \av{q}^2
$$
where $ \frac{1}{2}\alpha \av{q}^2$ represents the (linearized) energy of the fluid under the object (up to terms that are constant since the solid is not moving), while $E^{\rm ext}_\uU(U)$ corresponds to the full (linearized) energy of the fluid in the exterior domain,
$$
E^{\rm ext}_\uU (U)= \frac{1}{2} \int_\cE \big((1+ \varepsilon c'(\utheta))\theta^2+ q^2 +\delta^2 (\dx q)^2\big) dx.
$$

\begin{proposition}\label{propL2}
Let $T>0$ and assume that $\uU\in W^{1,\infty}([0,T]\times \cE)$ satisfies $\jump{\uq}=0$ and  that there are  constants $0<c_0\leq C_0$ and $m>0$ such that
\begin{equation}
c_0 \le  1+\eps c'(\utheta)\le C _0 , \qquad | \utheta, \dt \utheta , \dx \uq | \le m
\qquad \mbox{on } [0,T]\times \cE.
\end{equation}
Then there exists $\gamma=\gamma(m ,\frac{1}{c_0})$ such that the solutions  $U\in C^1([0,T];{\mathbb H})$ of  \eqref{Bousstheta_nh}--\eqref{CBl2_nh} satisfy  for $0\leq t\leq T$
$$
 E^{\rm tot}_\uU\big(U(t),\av{q(t)}\big)\leq e^{\eps \gamma t}\Big( E^{\rm tot}_\uU\big(U_{\vert_{t=0}},\av{q}_{\vert_{t=0}}\big)+\frac{\eps}{2}\int_0^t e^{-\eps \gamma s} \Vert (f,g)(s)\Vert^2_{L^2(\cE)}ds\Big).
$$
\end{proposition}
\begin{proof}
Multiplying the first equation of  \eqref{Bousstheta_nh} by $\theta$ and the second by $q$, one gets
$$
\dt e_\uU +\dx {\mathcal F}_\uU=\eps f\theta+\eps gq+\frac{\eps }{2}\dt (c'(\utheta))\theta^2+\eps (\dx \uq)q^2
$$
with
$$
e_\uU=\frac{1}{2}(1+ \varepsilon c'(\utheta))\theta^2+\frac{1}{2}q^2 +\frac{1}{2}\delta^2 (\dx q)^2
\quad\mbox{ and }\quad
{\mathcal F}_\uU=q \big( \theta+\eps \uq q-\delta^2 \dx\dt q\big).
$$
Integrating over $\cE$, we then get
\begin{align*}
\frac{d}{dt}E^{\rm ext}_\uU(U)- \jump{{\mathcal F}_\uU}\leq& \eps \Vert(f,g)\Vert_{L^2(\cE)}\Vert(\theta,q)\Vert_{L^2(\cE)}\\
&+ \frac{1}{2} \eps \max\Big\{ \dt c'(\utheta) ,2 \dx \uq  \Big\} \Vert(\theta,q)\Vert^2_{L^2(\cE)}
\end{align*}
Because  $\jump{q}=\jump{\uq}=0$,  \eqref{CBl2_nh} implies that
$$
\jump{{\mathcal F}_\uU} = \av{q} \big( \jump{\theta} - \delta^2 \jump{\dt \dx q} \big)  
= -\frac{1}{2}  \alpha \dt \av{q}^2.
$$
Therefore 
\begin{equation} 
\label{estL2}
\frac{d}{dt} E^{\rm tot}_\uU(U,\av{q}) \leq \frac{\eps}{2}\Vert(f,g)\Vert_{L^2(\cE)}^2 
+ \frac{1}{2}\eps \gamma   \Vert U \Vert^2_{L^2(\cE) } .
\end{equation}
$$  
\gamma = 
\max\{1, \| \dt c'(\utheta) ,2 \dx \uq   \|_{L^\infty}  \}
$$ 
Estimating $\Vert U \Vert^2_{L^2(\cE) }$ by $\frac{1}{c_0} E_\uU ^{\rm tot} (U, \av{q})  $, the result follows  from Gronwall's lemma.
\end{proof}


\subsection{Nonlinear estimates} 

In preparation for the control of the commutators in the energy estimates for the times derivatives, we first recall some  Gagliardo-Nirenberg inequalities. 
\begin{lemma} 
Let $\theta \in L ^\infty (\RR \times \mathcal E) $ such that $\dt^k \theta \in L^2 (\RR \times \mathcal E) $ for some 
$k > 0$. Then, for $0 \le j \le k$ and $p$ such that $ \frac{2}{p} = \frac{j}{k} $, there is a constant $C$ such that
$$
\big\| \dt^j \theta \big\|_{L^p} \le C  \big\| \theta \big\|^{1 - \frac{2}{p}}_{L^\infty}  \big\| \dt^k \theta \big\|^{\frac{2}{p}}_{L^2} .
$$
\end{lemma}
Combined with H\"older's estimates to interpolate between $L^2$ and 
$L^{2k/j}$, this implies that for  $ 1 \ge  \frac{2}{p} \ge \frac{j}{k} $
\begin{equation}
\label{GNE2}
\big\| \dt^j \theta \big\|_{L^p} \le C  \big\| \theta \big\|^{1 - \frac{2}{p}}_{L^\infty} 
\big( \sum_{l \le k}  \big\| \dt^l\theta \big\|_{L^2} \big) ^{\frac{2}{p}}  . 
\end{equation}

Functions defined in $ ] - \infty , T] \times \mathcal E$  can be extended  to $\RR \times \mathcal E$ by multiple reflections which preserve, up to numerical constants, the $L^p$ norms of the functions and their derivatives. Thus the estimates above remain true, with new constants $C$ for functions supported in 
$]- \infty , T]$. For functions supported in 
$[0, T]$, one can proceed similarly when $T \ge 1$ using  extension operators, but when $T \le 1$, one has to argue differently  
because the extension operators are not uniformly bounded when $T \to 0$.  
When $\theta \in L ^\infty ([0,T] \times \mathcal E) $ and  $\dt^k \theta \in L^2 ([0;T] \times \mathcal E) $, one can define the trace of $\dt^l \theta $ at $t = 0$ for $l < k$. If
\begin{equation}
\label{zero}
\dt^l \theta _{| t = 0}  = 0  , \qquad 0\le l < k
\end{equation}
the extension of $\theta$ by $0$ for $t < 0$ belongs to 
$ L ^\infty (] -\infty ,T] \times \mathcal E) $ and  $\dt^k \theta \in L^2 (] - \infty;T] \times \mathcal E) $ and the estimates above are satisfied. Suppose more generally that
\begin{equation}
\label{GNtraces} 
\dt^l \theta _{| t = 0} \in H^{1} (\mathcal E)  \qquad 0\le l < k. 
\end{equation}
Let 
 $$
 \tilde \theta = \sum_{l < k} \frac {t^l}{l !}  \dt^l \theta _{| t = 0}  \in C^\infty (\RR , H^1( \mathcal E)). 
 $$ 
Then  $\dt^l \tilde \theta _{| t = 0}  = \dt^l \theta _{| t = 0} $ and 
$\dt^l  \tilde \theta \in C^0(\RR \times H^1 (\mathcal E)) $ for $l < k$.  Because $H^1  \subset L^2 \cap L ^\infty
\subset L^p$ for $p \ge 2$, one has 
$$
\big\| \dt^j \tilde \theta \big\|_{L^p ([0, T] \times \mathcal E) } \le C T^{\frac{1}{p} } K_0, \qquad 
K_0 = \sum_{l < k}  \| \dt^l \theta _{| t = 0}\|_{H^{1} (\mathcal E) } . 
$$
The difference $\theta - \tilde \theta$ satisfies \eqref{zero} and we can estimate the $L^p$ norm of 
$\dt^l (\theta - \tilde \theta)$ using \eqref{GNE2}, as explained above. Adding up, this implies that
$$
\big\| \dt^j \theta \big\|_{L^p ([0, T] \times \mathcal E) } \le C T^{\frac{1}{p} } K_0 
+ C  ( \| \theta \|_{L^\infty} + K_0)^{1- \frac{2}{p} }  ( T^{\frac{1}{2} }K_0 + \sum_{l \le k}  \| \dt^l \theta \| _{L^2}) ^{\frac{2}{p}} 
$$

Using this estimate when $T < 1$ and  \eqref{GNE2} when $T \ge 1$, we have proved the following estimates which are valid in both cases. 
\begin{lemma}
\label{GNET}
 For all $T > 0$,  $0 \le j \le k$ and  $ 1 \ge  \frac{2}{p} \ge \frac{j}{k} $, there is $C$ such that 
 for all $\theta \in L ^\infty ([0,T] \times \mathcal E) $ with $\dt^k \theta \in L^2 ([0;T] \times \mathcal E) $
 satisfying  \eqref{GNtraces},   one has 
\begin{equation*}
\begin{aligned}
\big\| \dt^j \theta& \big\|_{L^p([0,T] \times \mathcal E)} \\
 & \le C  
\big( K_0  + \big\| \theta \big\|_{L^\infty([0,T] \times \mathcal E)} \big)^{1 - \frac{2}{p}}
\big( K_0 T^{1/2}   +   \sum_{l \le k}  \big\| \dt^l\theta \big\|_{L^2([0,T] \times \mathcal E)} \big) ^{\frac{2}{p}}
\end{aligned}
\end{equation*}
where 
$$
 K_0 = \sum_{1 \le l < k}  \| \dt^l \theta _{| t = 0}\|_{H^{1} (\mathcal E) } . 
$$
\end{lemma}

\subsection{Bounds on the time derivatives of the solution}

Let 
$U \in C^\infty ([0, T], \HH^n)$   be a solution of the transmission problem \eqref{Boussthetanl}--\eqref{Transs2}.  
Let also $U_j=(\theta_j,q_j)= (\partial_t^j \theta, \partial_t^j q)= \dt^j U$. Differentiating the equations in time
yields 
\begin{equation}\label{Boussthetaj}
\begin{cases}
(1+\eps c'(\theta))\dt \theta_j+ \dx q_j=\eps f_{(j)},\\
[1- \delta^2 \dx^2] \dt q_j+2\varepsilon q\dx q_j +\dx \theta_j =\eps g_{(j)}
\end{cases}
\quad \mbox{ on }\quad {\mathcal E}
\end{equation}
with transmission conditions
\begin{align}
\label{CBl1j}
\jump{q_j}&=0,\\
\label{CBl2j}
-\delta^2\dt \jump{\dx q_j}+\jump{\theta_j}&=-\alpha \av{q_{j+1}}
\end{align}
and 
\begin{equation}\label{deffg}
f_{(j)}=-  \sum_{k=1}^j \binom{k}{j} \dt^k (c'(\theta)) \theta_{j+1-k}  
\quad\mbox{ and }\quad
g_{(j)}= -2\sum_{k=1}^j \binom{k}{j} q_k \dx q_{j-k}.
\end{equation}
The following two lemmas provide some control on the source terms $f_{(j)}$ and $g_{(j)}$.
\begin{lemma}
\label{lemfj}
For $ 1 \le j \le n+1$ one has  for all $t \in (0,  T] $: 
$$
\big\| f_{(j)} \big\|_{L^2([0, t] \times \mathcal E )}  \le  C (\| \theta, \dt \theta  \|_{L^\infty},  K_0)
  \big( K_0 T^{1/2} +  \sum_{1 \le l \le j} \big\| \dt^l \theta  \big\|_{L^2([0, t] \times \mathcal E )  }  \big)
$$
with $K_0$ given by $ K_0 = \sum_{1 \le l < j}  \| \dt^l \theta _{| t = 0}\|_{H^{1} (\mathcal E) } $.
\end{lemma}

\begin{proof} Note that
$f_{(j)}$ is a linear combination of terms
$$
f _* = c^{(\nu+1 )} (\theta) \dt^{l_1}  \theta  \ldots  \dt^{l_{\nu+1}} \theta  
$$
with $\nu\ge 1$, $1 \le l_r \le j $ and $\sum l_r = j+1$.  We apply Lemma~\ref{GNET} to $\dt \theta$ and $k = j-1$, to estimate  $  \dt^{l_r} \theta$ in $L^{p_r} $ with  exponents $p_r$ such that $  \frac{l_r-1}{j-1 } \le \frac{2}{p_r} \le  1 $ 
and $\sum \frac{2}{p_r} = 1$. Note that such a choice is possible because 
$\sum {(l_r-1)}  =  j + 1 - \nu - 1 \le j- 1$.   This implies that  
$$
\big\| f_*  \big\|_{L^2 ([0, t] \times \mathcal E )}   \le C(\Vert \theta \Vert_{L^\infty})
\big( K_0 + \big\| \dt \theta \big\|_{L^\infty} \big)^{\nu }
\big(  K_0 T^{1/2}  +   \sum_{1\le l \le j}  \big\| \dt^l\theta \big\|_{L^2([0,T] \times \mathcal E)} \big) 
$$
and adding such estimates one gets the result.
\end{proof} 

\begin{lemma}
\label{lemgj}
For $ 1 \le j \le n+1$ one has  for all  $t \in (0,  T] $: 
$$
\big\| g_{(j)} \big\|_{L^2([0, t] \times \mathcal E )}  \le  C (\|\theta  , \dt U  \|_{L^\infty},  K_0)
  \big( K_0  T^{1/2} +  \sum_{1 \le l \le j}  \big\| \dt^l U \big\|_{L^2([0, t] \times \mathcal E )  }  
\big) 
$$
with now  
\begin{equation}
\label{M0}
 K_0 = \sum_{1 \le l < j}  \| \dt^l U _{| t = 0}\|_{H^{1} (\mathcal E) } . 
\end{equation}
\end{lemma}

\begin{proof}
Using the first equation of \eqref{Boussthetaj} one can replace 
$\dx q_{j -k} $  
by $f_{(j-k)} - (1 + \eps c'(\theta) \dt^{j-k-1} \theta $ in \eqref{deffg}. Using the representation of the 
$f_{(j-k)} $ given above, we see that    $g_{(j)}$ is a linear combinations of  two types of terms: first we have 
\begin{equation}
\label{termgj}
g_* =    c_* (\theta)  \dt^k q \dt^{l_1}  \theta  \ldots  \dt^{l_{\nu+1}} \theta 
\end{equation}
with  $1\le l_r \le j-k$ with $\sum_{r=1}^{\nu+1} l_r=  j - k + 1 $  and  second
$$
(1 + \eps c'(\theta) ) \dt^k q  \dt^{j-k-1} \theta
$$
with $ 1 \le  k \le j $.  
 The estimate   follows from Lemma~\ref{GNET} applied to  $\dt q$ and $\dt \theta$, the 
 $L^\infty$ norm of the coefficients $c_*(\theta)$ depending only on $\| \theta \|_{L^\infty}$. 
\end{proof} 

We can now estimate the $L^2 $ norm of $U_j$ applying Proposition~\ref{propL2} to \eqref{Boussthetaj}.  
Let 
\begin{equation}
\label{mfE}
\mfE  (t) = \sum_{0 \le  j \le n+1} \Big(  \big\| \dt^j U \big\|^2_{L^2(\mathcal E)}    
+ \delta^2  \big\| \dt^j \dx q \big\|^2_{L^2(\mathcal E)} + \alpha |  \av{\dt^j q} |^2 \Big). 
\end{equation}

\begin{proposition}\label{Estdtj}
Let $n\in {\mathbb N}$ and $T>0$, and assume that 
$U \in C^\infty ([0, T]; \HH^n) $ is a solution of \eqref{Boussthetanl} such that  there are constants $0<c_0\leq C_0$ and $0<m$ such that
$$
c_0 \le  1+\eps c'(\theta)\le C _0  \quad and \quad | \theta, \dt \theta ,\dt q | \le m  \qquad on \ [0, T] \times \mathcal E. 
$$
Then there are constants $C = C (C_0, c_0)$ and $\gamma  = \gamma (K_0, m ,  c_0^{-1}, C_0) $ such that 
$$
 \mfE(t)  \leq e^{\eps \gamma t}\big( C \mfE (0)  
     +  \eps \gamma \, t\big),
 $$
with $K_0$ is given by \eqref{M0} {\rm (}with $j=n+1${\rm )}.   

\end{proposition}

\begin{proof}
By ~\eqref{estL2}, one has 
$$
\begin{aligned}
 E^{\rm tot}_U\big(U_j(t),\av{q_j (t)}\big)\leq  
 E^{\rm tot}_U\big(U_j (0) ,\av{q(0)} \big)
 +\frac{\eps}{2} \int_0^t \Vert (f_{(j)},g_{(j)} )(s)\Vert^2_{L^2(\cE)}ds \\
 + \frac{1}{2}\eps \gamma_1 \int_0^t \Vert U_j  (s)\Vert^2_{L^2(\cE)}ds . 
\end{aligned}
$$
with  $\gamma_1 $ depends on $m$ and $\| \dx q \|_{L^\infty}$. By the equation 
we can estimate
$$
| \dx q  |  \le C_0 | \dt \theta | , 
$$
and thus $\gamma_1 \le \gamma (m, C_0) $.  
Using the lemmas above to estimate the $L^2$ norms of $f_{(j)}$ and $g_{(j)}$ and the inequalities
$$
\tilde c_0  \big\| \dt^j U \big\|^2_{L^2}    
+ \delta^2  \big\| \dt^j \dx q \big\|^2_{L^2}  \le 2 E^{\rm ext} (U_j) \le  \tilde C_0 \big\| \dt^j U \big\|^2_{L^2}    
+ \delta^2  \big\| \dt^j \dx q \big\|^2_{L^2}
$$
with $\tilde c_0 = \min\{1, c_0\} $, $\tilde C_0 = \max \{1, C_0\}$,  
one obtains that
$$
\tilde c_0 \mfE  (t) \le \tilde C_0  \,  \mfE(0) + \eps \gamma (K_0, C_0,  m )  
 \Big(t +   \int_0^t \mfE  (s) ds \Big) .
$$
The proposition follows by Gronwall's lemma. 
\end{proof}


\section{Estimates of  $x$-derivatives}
\label{dxD}

The next step is to prove $L^\infty $ estimates  for $\theta$ and $\dt U$ that are required in the statement of Proposition \ref{Estdtj}.  They  follow from $H^1$ estimates, 
that is from $L^2 $ estimates   of $(\dx \theta, \dt \dx  U)$.  The quantity $\dx q$ is explicitly given by  the first equation
but the control of $ \dx \theta$ is more involved because of the term 
$\dx^2 \dt q$ in the second equation. We will derive a second order o.d.e. for $\dx \theta$, and get estimates from it. 
 It turns out that to obtain closed estimates 
we need to control  $\dx \dt^j \theta $ up to $j = 4$.

\subsection{Estimates for $q$}

We show here that the $H^1$-norm (and therefore the $L^\infty$-norm) of $\dt^j q$ can be controlled up to $j=n$ in terms of the energy norm, the initial data, and the $L^\infty$ norm of $\theta$ and $\dt \theta$.
\begin{lemma}
\label{lemdxq} 
Let $n\in {\mathbb N}$ and $t>0$, and suppose that $U \in C^\infty ([0, T]; \HH^n)$ where $\HH^n = H^{n+1}\times H^{n+2}({\mathcal E})$ is a solution of \eqref{Boussthetanl} 
such that  there are constants $0<c_0\leq C_0$ and $0<m$ such that
\begin{align*}
& 
c_0 \le  1+\eps c'(\theta)\le C _0  \\
 &| \theta, \dt \theta  | \le m  \qquad on \ [0, T] \times \mathcal E, 
\\
&\mfE (t)  \le M  \ \ \qquad for  \  t \in [0, T]
\end{align*}
for some positive constants $m$ and $M$ with $\mfE$ given by \eqref{mfE}.
Then there are constant $C_1(m)$ and $C_2(m,K_0)$ such that 
$$
\| \dx q \|_{L^\infty[0,T ] \times \cE)} \le C_1(m) 
$$
and,  for $ 0  \le j \le n$ , 
$$
\big\| \dt ^j q (t ) \big\|^2_{H^1(\cE)}   \le  C_2(C_0,K_0)\big(M+\mfE(0)\bigr)
     + C(m,K_0)(1+M) \epsilon^2 t^2,  
$$
where $K_0$ is defined at \eqref{M0}. 
\end{lemma}

\begin{proof}
The first estimate is immediate from the equation: 
$$
 \dx q  =  - (1+\eps c'(\theta))\dt \theta. 
$$
The $L^2$ norm of $ \dt^j  q$ is controlled by $M^{1/2}$. Thus, to prove the second inequality, it is sufficient to prove 
an $L^2$ estimate of $ \dt^j \dx q$. By \eqref{Boussthetaj} 
$$
 \dt^j \dx q = - (1+\eps c'(\theta))\dt ^{j+1} \theta + \eps f_{(j)} .
$$
The $L^2$-norm of the first term is of order $C_0 M^{1/2} $. For the second we use that 
$$
 \|   f_{(j)} (t) \|^2_{L^2 } \le  2\|   f_{(j)} (0) \|^2 _{L^2 } + 2 t \int_0^t  \| \dt f_{(j)} (t) \|^2_{L^2} ds . 
$$
Using Lemma~\ref{GNET}  one obtains that for $j \le n$, 
\begin{align*}
\int_0^t \| \dt f_{(j)} (t) \|^2_{L^2} ds &\le C(m, K_0) \big(\sum_{j \le n+1} \int_0^t  \big\| \dt^j \theta (t)  \big\|^2 _{L^2(\cE)} + t\big)\\
&\le C(m,  K_0)  (1+M )t. 
\end{align*}
Hence 
$$
\big\| \dt ^j \dx q (t ) \big\|^2_{L^2(\cE)}  \le C_0\,  M +  2\eps^2 \|   f_{(j)} (0) \|^2 _{L^2 } + C(m,  K_0)   (1+M)\eps^2 t^2.
$$
Noticing that   $ \|   f_{(j)} (0) \|^2 _{L^2 } \le C(K_0)\mfE(0) $, the lemma follows. 
\end{proof} 

\subsection{A linear  o.d.e.  for $\dx \theta_j$} 

In the second equation of \eqref{Boussthetanl} we use the first one to replace  $\dx^2 \dt q $ 
by $- \dt \dx  (1 + \eps c'(\theta) \dt \theta $, obtaining that 
\begin{equation*}
\label{ode1}
\dx \theta + \delta^2 \dt \dx \big( (1+\eps c'(\theta))\dt \theta \big)  =  - \dt q  +   2 \eps q  (1+\eps c'(\theta))\dt \theta.  
\end{equation*} 
We reorganize this equation using that 
$$
\dx \bigl[(1+\eps c'(\theta))\dt \theta \bigr] = (1+\eps c'(\theta) \dt Y + \eps  Y  \dt c'(\theta)    ,  
$$
with $Y= \dx \theta $ so that
$$
\dt \dx \bigl[(1+\eps c'(\theta))\dt \theta \bigr]= (1+\eps c'(\theta) \dt^2 Y + 2 \eps   \dt c'(\theta)  \dt Y + \eps \dt^2 c'(\theta) Y    ,  
$$
and $Y$ appears as a solution of the equation 
\begin{equation}
\label{ode0}
a_0 \delta^2 \dt^2 Y + \eps \delta^2 a_1 \dt Y + (1 +\eps \delta a_2)  Y = \chi + \eps\psi
\end{equation}
where 
\begin{equation}
\label{aedo}
a_0 = 1 + \eps c' , \quad a_1 =  2 \dt c' , \quad a_2 =    \delta  \dt^2 c' , \quad c' = c'(\theta), 
\end{equation}
and 
$$
 \chi =  - \dt q , \qquad \psi =   2  q  (1+\eps c'(\theta))\dt \theta.
$$

We differentiate \eqref{ode0} in time (alternately, on can start from \eqref{Boussthetaj}) to obtain an equation for $ Y_j = \dx \theta _j$. One has
$$
\begin{aligned}
&\dt^j (a_0 \dt^2 Y) = a_0 \dt^2 Y_j +  \eps j \dt c' \dt Y_j +\frac{\varepsilon}{2} j (j-1) \dt^2 c' Y_j  +\eps   \sum_{k = 3}^j  \big(^j_k \big)  \dt^k (c'(\theta) )  Y_{j +2 - k} 
\\
&\dt^j (a_1 \dt Y) = a_1 \dt Y_j +     2j \dt^2 c'   Y_j  +  2\sum_{k = 2}^j  \big(^j_k \big)  \dt^{k+1} c'  Y_{j +1 - k} 
\\
&\dt^j ( a_2 Y)  = \delta \Big( \dt^2 c'  Y_j +   \sum_{k = 1}^j  \big(^j_k \big)    \dt^{k+2} c'  Y_{j-k} \Big). 
\end{aligned}
$$
Thus 
\begin{equation}
\label{odej}
a_0 \delta^2  \dt^2  Y_j + \eps \delta^2 a_{1, j}   \dt Y_j  +   (1 + \eps \delta a_{2, j}  )  Y_j   =   
 \chi_j  + \eps  \psi _{(j)}   + \eps \delta \varphi_{(j)}
\end{equation}
with 
\begin{equation}
\label{aedoj} 
a_{1, j} = (j+2) \dt c' , \quad a_{2, j} = \frac{1}{2} ( j+1) (j+2)  \delta  \dt^{2} c' , 
\end{equation}
$$
\chi_j = - \dt^{j+1} q , \quad \psi_{(j)} = 2 \dt^j \big(  q (1 + \eps c') \dt \theta \big)  
$$
and 
$$
\varphi_{(j)} =  \sum_{k = 0 }^{j-1}   \beta_{k,j}  \delta    \dt^{j+2 - k}   (c'(\theta) ) \  Y_{ k } ,
$$
where the $\beta_{k,j} $ are numerical constants of no importance. 

For our purposes, we are looking at 
solutions which remain uniformly bounded (with respect to $\delta$ and $\eps$) for times of order $O(\eps^{-1})$. Since the ODE is singular  because of the coefficient  $\delta^2$ in front of the higher order term, a direct application of Duhamel's formula requires that the  right-hand-side is of order $O(\eps\delta)$ in order for the solution to be uniformly bounded solution on the $O(\eps^{-1})$ time scale. A refined study, that takes advantage of the oscillating nature of the solutions of the homogeneous equation, is required to handle the contribution of $O(\eps)$ and even $O(1)$ source terms. This is shown in the next subsection. 


\subsection{The basic estimate for the o.d.e}
We start with a simplified equation 
\begin{equation}
\label{sode}
a_0  \delta^2    \dt^2  Y   +   Y =   
\eps \delta \varphi + \eps\psi+\chi
\end{equation} 
 assuming  that the coefficient $a_0$ satisfies on $[0,T]\times \cE$
\begin{equation}
\label{odeb0}
0 <   c_0  \le a_0 \le C_0,   \quad  |  \dt a_0 |  \le  \eps m . 
\end{equation}

We are looking for  estimates of $Y(t)$ in $L^2$ and $L^\infty$. To avoid repetitions  we make a unique statement, 
using the notation $\BB$ for $L^2(\mathcal E)$ or $C^0 \cap L^\infty(\mathcal E)$. 
 
\begin{lemma}\label{lemsode}
Given constants $c_0$ $C_0$ and $m$, there are $C = C(c_0^{-1}, C_0) $ and  
$\gamma = \gamma (c_0^{-1},  m)$ such that 
for $T>0$, $\eps, \delta$  in $(0, 1] $,  $a_0$  satisfying  \eqref{odeb0}  
  and for  $\varphi \in L^\infty([0,T];{\mathbb B})$,
 $\psi \in W^{1,\infty}([0,T]; \BB )$ and $\chi \in W^{3,\infty}([0,T]; \BB )$, 
 the solutions to the o.d.e.  \eqref{sode} 
satisfy  the following estimate for all $0\leq t \leq T$,
\begin{equation}
\label{estode}
\begin{aligned}
\Vert Y(t),& \delta  \dt Y(t)  \Vert^2_\BB  
\\
&  \leq   C e^{(\eps+\delta^2) \gamma  t}
 \big(  \Vert \big (Y(0),\delta \dt Y(0))\Vert^2_{\BB} +  \eps \| \psi(0) \|^2_\BB + \| \chi(0)\|_{\BB}^2  +  \mfs (t) \big)
\end{aligned}
\end{equation}
with
\begin{equation}\label{defs}
\begin{aligned}
\mfs (t)  =&  \eps  \int_0^t  \| \varphi  (s), \psi (s), \dt \psi (s)  \|_{\BB}^2 {\rm d} s  \\
&+ \delta^2  \int_0^t  \Big(    \|  \dt^3 \chi (s) \|^2_{ \BB} + \eps m  \|  \dt^2 \chi(s)  \|^2_{ \BB } 
 + \eps^2 m^2  \|  \dt  \chi (s)  \|^2_{ \BB }\Big)  {\rm d}s 
 \\
 & +   \|\chi(t),  \delta \dt \chi (t) ,   \delta^2  \dt^2 \chi (t) \|^2_{ \BB} . 
\end{aligned}
\end{equation}
\end{lemma}

\begin{proof}[Proof when $\BB = L^2$]The natural energy is 
\begin{equation}
I =\frac{1}{2}\int_{\mathcal E}  \Big( a_0   ( \delta \dt Y \big)^2 +   Y ^2 \Big) dx  
\end{equation}
which satisfies 
$$
\tilde c_0  \big(  \| Y(t) \|_{L^2}^2 + \| \delta \dt Y(t) \|_{L^2}^2\big) 
 \le I(t) \le \tilde C_0   \big(  \| Y(t) \|_{L^2}^2 + \| \delta \dt Y(t) \|_{L^2}^2\big) 
$$
where $\tilde c_0 = \frac{1}{2}\min \{1, c_0\} , \tilde C_0 =   \frac{1}{2} \max\{1, C_0\}$.

 Multiplying the equation by $ \dt Y  $  and integrating over $ [0, t] \times \mathcal E$  yields
 \begin{equation} 
 \label{estI}
 \begin{aligned}
 I(t)  & =  I( 0) +  \frac{1}{2}   \int_{[0, t] \times \mathcal E } ( \dt a_0) (\delta  \dt Y)^2 
 +  \int_{[0, t] \times \mathcal E}   \Phi  \dt Y  
 \\
& = I(0) +  I_1(t) + I_2(t), 
\end{aligned}
 \end{equation}
 where $ \Phi  =  \eps \delta \varphi + \eps\psi+\chi$. 
Because  $| \dt a_0 | \le  \eps m $,  the first integral  $I_1$ in the right hand side can be estimated by 
\begin{equation}
\label{I1}
I_1(t) \le   \frac{1}{2}   \eps \tilde c_0^{-1}  m \int_0^t I(s) ds. 
\end{equation}

 In order to prove the lemma,  we note that
$Y$ is the sum of the three  solutions obtained  first when  $\Phi = \eps \delta \varphi $, 
second when $\Phi = \eps  \psi  $ and the initial conditions vanish and  third when 
$\Phi = \chi$ and the initial conditions vanish. We prove the estimate in  each case separately. 

\medskip

\noindent -- {\bf Case a) : } $\Phi = \eps \delta \varphi$.  In this case, the second integral $I_2$ is 
 $$
 \begin{aligned}
 I_2 (t)  	& =  \eps \int_{[0, t] \times \mathcal E}   \varphi      \delta \dt Y 
 \le \eps  \| \varphi (s) \|_{L^2([0, t] \times \cE)}     \|  \delta  \dt Y \|_{L^2([0, t] \times \cE)}
\\
&  \le 
 \frac{\eps }{2}\Big( \|  \delta \dt Y \|_{L^2([0, t] \times \cE)}^2 +  \| \varphi  \|_{L^2([0, t] \times \cE)} ^2 \Big)
  \\
& \le \frac{1}{2} \eps \tilde c_0^{-1}  \int_0^t I(s) ds +  \frac{1}{2} \eps \int_0^t  \| \varphi  \|^2_{L^2(\cE)}\, ds . 
  \end{aligned}
 $$
 Using this last estimate with \eqref{estI} and \eqref{I1} and Gronwall's lemma, this implies the estimate \eqref{estode} with $\psi=\psi=0$.

\medskip

\noindent  -- {\bf Case b) : } $\Phi = \eps \psi $ and $Y_{| t = 0} =\dt Y_{| t = 0} = 0$. 
 Then
 $$
 I_2 (t)
   =  \eps  \int_{[0, t] \times \mathcal E}  \psi  \dt Y 
  = \eps  \int_{\mathcal E}   ( \psi  Y)   (t)  dx 
 -  \eps \int_{[0, t] \times \mathcal E}   Y \dt \psi   . 
 $$
 The second term satisfies 
  $$
\Big|\eps \int_{[0, t] \times \mathcal E}   Y \dt \psi \Big|   \le  \eps \tilde c_0^{-1}    \int_0^t I(s) ds  +
  \frac{1}{2} \eps \int_0^t   \| \dt \psi (s)  \|^2_{  L^2}  ds . 
 $$
Moreover for all $\kappa \in (0, 1]$ 
 $$
\Big|   \int_{\mathcal E}   ( \psi  Y) (t) \Big|  \le  \frac{\kappa}{2 c_0}   I(t)  
+  \frac{1}{2 \kappa}  \|\psi (t)\|^2_{L^2(\mathcal E)} 
 $$
and 
 $$
 \begin{aligned}
 \|\psi (t)\|^2_{L^2(\mathcal E)} & =  \|\psi (0)\|^2_{L^2(\mathcal E)} 
    + 2 \int_{[0, t] \times (\mathcal E)} \psi \dt \psi \\
 & \le  \|\psi (0)\|^2_{L^2(\mathcal E)} +   \int_0^t  ( \| \psi (s)  \|^2_{L^2(\mathcal E)}  
 +  \| \dt \psi (s)  \|^2_{  L^2(\mathcal E)})   ds
  \end{aligned}
 $$
Therefore 
$$ 
\begin{aligned}
 I_2 (t) \le     \frac {\eps \kappa}{2 \tilde c_0}   I(t) &  + \eps  \tilde c_0^{-1}    \int_0^t I(s) ds     \\ 
 &  +       \frac{\eps}{2\kappa}  \| \psi(0)\|^2_{L^2}    +  \frac{\eps }{2 }(\frac{1}{\kappa}+1)
   \int_0^t   ( \| \psi (s)  \|^2_{L^2}  +  \| \dt \psi (s)  \|^2_{  L^2})  ds  
 \end{aligned}
 $$
 We use this estimate together with \eqref{I1} and \eqref{estI}. 
  Choosing $\kappa $ a small fraction of $\tilde c_0$ to absorb   the term in $I(t)$ from the left to the right, and using Gronwall's lemma, the estimate \eqref{estode} follows. 
 
\medskip

\noindent -- {\bf Case c) : } $\Phi = \chi $ and $Y_{| t = 0} =\dt Y_{| t = 0} = 0$. Then 
 \begin{equation}\label{I2chi}
 I_2 (t)
 =     \int_{[0, t] \times \mathcal E}  \chi  \dt Y 
  =    \int_{\mathcal E}   ( \chi   Y)   (t)  dx 
 -    \int_{[0, t] \times \mathcal E}   Y \dt \chi   =  J_1 + J_2
 \end{equation}
Let us first bound $J_1$ as follows
 \begin{equation} \label{J1}
   \int_{\mathcal E}   ( \chi   Y)   (t)  dx  \le \kappa \, I(t) + \frac{1}{c_0 \kappa} \|\chi(t)\|^2_{L^2(\mathcal E)}.
 \end{equation}
 The difference with the previous case is that there is no  $\eps$ in  so we have to treat the 
second term differently.  Using the equation, we note that 
 \begin{equation}
 \label{truc}
 J_2 =  -   \int_{[0, t] \times \mathcal E}   Y \dt \chi  =  - 
  \int_{[0, t] \times \mathcal E}  \chi  \dt  \chi  + 
    \delta^2  \int_{[0, t] \times \mathcal E} a_0   \dt \chi       \dt ^2 Y 
    = J_{2,1} + J_{2,2} 
\end{equation}
 The first term is
\begin{align}
\label{J21}
J_{2,1}=  -  \int_{[0, t] \times \mathcal E} \chi \dt  \chi  
=  \frac{1}{2} \big(  \| \chi(0) \|_{L^2}^2 -   \| \chi (t) ) \|_{L^2}^2 \big) \le  \frac{1}{2}    \| \chi(0) \|_{L^2}^2 . 
\end{align}
 
In the  second  term, we integrate by parts  : 
$$
\begin{aligned}
J_{2,2}=  &   \delta^2   \int_{[0, t] \times \mathcal E}  a_0   \dt  \chi     \dt^2  Y  
\\ 
= &  -  \delta^2   \int_{[0, t] \times \mathcal E}     a_0 \dt^2  \chi  \dt Y 
-  \delta^2  \int_{[0, t] \times \mathcal E}    \dt  a_0    \dt  \chi   \dt Y    
+ \delta^2  \int_{\mathcal E} a_0 \dt\chi \dt Y (t)  . 
\end{aligned}
$$
In the first integral, we integrate by parts again and get that 
\begin{equation}
\begin{aligned}
\label{J22}
&& J_{2,2} =   \Bigl[\delta^2   \int_{[0, t] \times \mathcal E}    \dt (a_0 \dt^2  \chi)   Y -   \delta^2  \int_{[0, t] \times \mathcal E}    \dt  a_0    \dt  \chi   \dt Y  \Bigr]
\\ 
 && + \Bigl[
 \delta^2  \int_{\mathcal E} (a_0 \dt \chi \dt Y) (t)  -  \delta^2  \int_{\mathcal E} (a_0 \dt^2 \chi   Y) (t) 
\Bigr] \\
&& =  \int_{[0,t]\times \mathcal E} F_1 + \int_{\mathcal E} F_2(t) 
\end{aligned}
\end{equation}
 The integrals over $[0, t] \times \mathcal E$ are dominated by
 \begin{equation}
 \begin{aligned}
 \label{FF1}
 \int_{[0,T]\times \mathcal E} F_1
 &  \le   
 \delta^2 \tilde c_0^{-1}  \int_0^t I(s) ds  + \frac{1}{2}  \delta^2 \int_0^t   \| \dt (a_0 \dt^2 \chi) \|^2_{L^2 }
  \\  
 & \hskip5cm + \frac{1}{2}  \delta^2 \eps^2  m^2  \int_0^t \|  \dt  \chi \|^2_{ L^2 }  . 
 \end{aligned}
 \end{equation}
The integrals over $\mathcal E$  at time $t$ are estimated by   
\begin{equation}
\begin{aligned}
\label{FF2}
\int_{\mathcal E}F_2(t) \le 
      \tilde C_0 \big(   \| \delta \dt\chi (t)   \|_{L^2} \|  \delta \dt Y (t) \|_{L^2} +  
 \| \delta^2  \dt^2 \chi (t)   \|_{L^2} \|  Y (t) \|_{L^2}\big) 
 \\
   \le \frac{\kappa}{2\tilde c_0} I(t)  + \frac{\tilde C_0^2 }{ 2 \kappa} \big(  \| \delta \dt\chi  (t)   \|^2_{L^2} +  \| \delta^2 \dt^2\chi  (t)   \|^2_{L^2}\big)
\end{aligned}
\end{equation}
We can therefore conclude using \eqref{I2chi}--\eqref{FF2} 
$$
\begin{aligned}
| I_2  (t) |     \le  (\frac{\kappa}{2\tilde c_0}  + \kappa) I(t)   +    & \delta^2 \tilde c_0^{-1}   \int_0^t I(s) ds 
\\ &  +   \frac{1}{2} \| \chi (0)\|^2_{L^2}  + \frac{\tilde C_0}{2\kappa}    \| (  \delta \dt \chi (t) ,   \delta^2  \dt^2 \chi (t) )\|^2_{ L^2}  +  \frac{1}{c_0 \kappa} \|\chi(t)\|^2_{L^2}
\\
&  +  \tilde C_0   \int_0^t  \Big( \delta^2   \|  \dt^3 \chi \|^2_{ L^2} + \eps  \delta^2 m \|  \dt^2 \chi \|^2_{ L^2} 
 + \eps^2  \delta^2  m^2 \|  \dt  \chi \|^2_{ L^2} \Big). 
\end{aligned}
$$
Gathering the estimates obtained for the three cases, and choosing $\kappa$ small enough, we obtain that there is $\gamma=\gamma(c_0^{-1},m)$ and $C = C(c_0^{-1}, C_0)$ such that 
$$
\begin{aligned}
| I  (t) |     \le \frac{1}{2}  I(t)   +    & \gamma  ( \eps + \delta^2)     \int_0^t I(s) ds 
\\ & 
 +  C\big( \eps \Vert \psi(0)\Vert_{L^2}+  \| \chi (0)\|^2_{L^2} \big) 
 + C    \| (\chi(t), \delta \dt \chi (t) ,   \delta^2  \dt^2 \chi (t) )\|^2_{ L^2} 
\\
&+ \eps  C \int_0^t \big( \Vert \varphi \Vert_{L^2} +
 \| \psi (s)  \|^2_{L^2}  +  \| \dt \psi (s)  \|^2_{  L^2}\big)  \\
&  +  \delta^2 C   \int_0^t  \Big(    \|  \dt^3 \chi \|^2_{ L^2} + \eps   m \|  \dt^2 \chi \|^2_{ L^2} 
 + \eps^2 m^2 \|  \dt  \chi \|^2_{ L^2} \Big). 
\end{aligned}
$$
Gronwall's lemma implies the result.   
\end{proof} 

\begin{proof}[Proof when $\BB = C^0 \cap L^\infty$]
The proof above applies for each fixed $x \in \cE$; taking the supremum in $x$ instead of integrating over $\cE$ therefore yields the result.  
\end{proof}

The complete  equation \eqref{odej}  reads 
\begin{equation}
\label{code}
a_0  \delta^2    \dt^2  Y   +   \eps \delta^2 a_1  \dt Y  + (1 + \eps \delta a_2) Y =   
\eps \delta \varphi + \eps\psi+\chi
\end{equation} 
and can be seen as a perturbation of \eqref{sode}. 

\begin{lemma}\label{lemcode}
Given constants $c_0$, $C_0$ and $m$, there are $C = C(c_0^{-1}, C_0) $, $\gamma = \gamma (c_0^{-1}, C_0, m)$ and a smooth nondecreasing function ${\mathfrak e}: \RR^+\to \RR^+$ with ${\mathfrak e}(0)=1$ such that 
for $T>0$, $\eps, \delta$  in $(0, 1] $,  $a_0$  satisfying  \eqref{odeb0}  and $a_1, a_2$ satisfying
\begin{equation}
\label{odeb1}
\|  a_1 \|_{L^\infty ([0,T] \times \mathcal E) }  \le m, \quad \|  a_2 \|_{L^\infty ([0,T] \times \mathcal E) }  \le m,
\end{equation}
  and for  $\varphi \in L^\infty([0,T];{\mathbb B})$,
 $\psi \in W^{1,\infty}([0,T]; \BB)$ and $\chi \in W^{3,\infty}([0,T]; \BB)$, 
 the solutions to the o.d.e \eqref{code} 
satisfies   estimate  
\begin{equation}
\label{estode2}
\begin{aligned}
\Vert Y(t&), \delta  \dt Y(t)  \Vert^2_\BB  
\\
&  \leq   C {\mathfrak e}\big( {(\eps+\delta^2) \gamma  t}\big)
 \big(  \Vert \big (Y(0),\delta \dt Y(0))\Vert^2_{\BB} +  \eps \| \psi(0) \|^2_\BB + \| \chi(0)\|_{\BB}^2  +  \mfs (t) \big),
\end{aligned}
\end{equation}
with ${\mathfrak s}(t)$ as in \eqref{defs}. 
\end{lemma}

\begin{proof}
We put the perturbation $ \eps \delta (a_2 \delta \dt Y + a_2 Y) = \eps \delta \phi $  on the right hand side, and use the estimate \eqref{estode}. When $\BB = L^2$, this terms contributes to the right hand side of the estimate adding a term
dominated by
$$
\eps  C e^{ \gamma (\eps + \delta^2) t} \int_0^t \| \phi (s) \|^2_{L^2}  {\rm d}s 
\le  \eps C m e^{ \gamma (\eps + \delta^2) t} \int_0^t (\| \delta \dt Y(s) \|^2_{L^2} + \| Y (s) \|^2 ){\rm d}s 
$$
which is absorbed from the right to the left by Gronwall's lemma, implying \eqref{estode2}.  When $\BB = C^0 \cap L^\infty$, there is a similar estimate for each fixed $x \in \cE$, and  one concludes taking sup 
as in the proof of the previous lemma. 
\end{proof}


\subsection{$L^\infty$ estimate of $Y = \dx \theta$} 

We turn back to \eqref{Boussthetanl} and consider a solution $U \in C^\infty ([0, T]; \HH^n) $ such that 
\begin{align}
\label{bound0}
 & 0 < c_0 \le 1 + \eps c'(\theta) \le C_0, \quad  | \theta,  \dt \theta, \delta \dt^2 \theta, q | \le m_1  
 \qquad on   \  [0, T]\times \cE, 
\\
\label{bounde}
& \sum_{j \le n+1} \big\| \dt^j \theta(t),\delta\dt^{j+1}\theta(t),\dt^j q (t) \big\|^2_{L^2(\cE)} \le M  \qquad for  \  t \in [0, T].
\end{align} 

With  $a _0 = 1 + \eps c'(\theta)$ and coefficients $a_{1}, a_{2} $ given by \eqref{aedo}, the conditions 
\eqref{odeb0} and \eqref{odeb1} are satisfied, with a constant,  $m= m(m_1) $ and Lemma \ref{lemcode} can be used to provide a bound on $Y=\dx\theta$.
\begin{lemma} 
\label{Yinfty}
There are $C = C(c_0^{-1}, C_0)$ and $\gamma = \gamma (c_0^{-1}, C_0, m_1)$ and a nondecreasing function ${\mathfrak e}:\RR^+\to \RR^+$ with ${\mathfrak e}(0)=1$ such that if 
$U \in C^\infty ([0, T]; \HH^n) $ is a solution of \eqref{Boussthetanl} satisfying \eqref{bound0}, then $Y = \dx \theta$  satisfies 
$$
\| (Y (t) ,\delta\dt Y(t)\|_{L^{\infty} }^2 \! \le \!  C {\mathfrak e}\big( { \gamma  (\eps + \delta^2) t }\big)
\Big(  \| Y (0),\delta\dt Y(0) \|^2_{L^{\infty} } + (1 + t( \eps + \delta^2) c(m_1))  \mfm (t) \Big) 
$$
where
\begin{equation}
\label{mfm} 
\mfm (t) = \sup_{0 \le s\le t} \Big(  \sum_{l \le 2} \| \dt^l \theta (s) \|^2_{L^\infty}  + \sum_{l \le 4} \| \dt^l  q  (s) \|^2_{L^\infty}\Big) .
\end{equation}
\end{lemma}

\begin{proof} Recalling that  $Y$ satisfies the o.d.e \eqref{ode0}, we use Lemma~\ref{lemcode} with $\BB = C^0 \cap L^\infty$ remarking that there is no $\varphi$ 
in the right and side. Then  we just have to control time derivatives of $\chi = - \dt q$ and $\psi = 2 q (1 + \eps c'(\theta)) \dt \theta $.  

\smallskip

\noindent Note that for $k \le 3$ one has 
$$
\| \dt^k \chi (t)  \|^2_{L^\infty} = \| \dt^{k+1} q  (t)  \|^2_{L^\infty} \le \mfm (t).
$$ 
Similarly,  
$$
\| \psi (t) , \dt \psi \|^2_{L^\infty}  \le C (m_1) \mfm(t) . 
$$
and 
$$
\eps \| \psi (0) \|^2_{L^\infty} + \eps \int_0^t \| \psi (s) , \dt \psi (s)  \|^2_{L^\infty}  ds \le 
C (m_1) ( \eps + \eps t)  \mfm(t) . 
$$
The estimate follows then using Lemma~\ref{lemcode}. 
\end{proof}


\subsection{$L^2$ estimates of $\dx \theta_j$} 
 
 Our goal here is to give an estimate of 
\begin{equation}
\label{mfE1} 
\mfE_1 (t)  = \sum_{j \le n-3} \big(  \| \dt^j Y (t)\|_{L^2 (\cE)}^2+   \| \delta \dt^{j+1} Y (t)\|_{L^2 (\cE)}^2\big). 
\end{equation} 
In order to use Lemma~\ref{GNET} we also introduce 
\begin{equation}
\label{M1}
K_0= \sum_{1 \le l < n+1}  \| \dt^l  U  _{| t = 0}\|_{H^{1} (\mathcal E) } 
\quad\mbox{ and }\quad
 K_1 = \sum_{1 \le l < n-3}  \| \dt^l \dx \theta  _{| t = 0}\|_{H^{1} (\mathcal E) }.
\end{equation}
As usual, ${\mathfrak e}$ denotes a nondecreasing function such that ${\mathfrak e}(0)=1$.
\begin{lemma}
\label{YL2}
There are constants $C = C(c_0^{-1}, C_0)$, $\gamma = \gamma (c_0^{-1}, C_0, m_1, K_0)$ and 
$M_1  = M_1 (m_1, K_0, K_1, M ) $ where $c_0,C_0,m_1,K_0,K_1$ are defined by \eqref{bound0}, \eqref{M1} and a nondecreasing function ${\mathfrak e}:\RR^+\to \RR^+$ with ${\mathfrak e}(0)=1$ such that if 
$U \in C^\infty ([0, T]; \HH^n) $ is a solution of \eqref{Boussthetanl} satisfying \eqref{bound0} and \eqref{bounde},
then
\begin{equation}
\label{blop}
\begin{aligned}
\mfE_1(t)  \le  C {\mathfrak e}\big({ \gamma  (\eps + \delta^2) t }\big) \Big( \mfE_1 (0) 
  +    (1 + ( \eps + \delta^2) t  )    (1 + \| Y \|_{L^\infty([0, t] \times \cE)}^2 )M_1   \Big)
\end{aligned}
\end{equation}
where $\mfE_1$ is given by \eqref{mfE1}.
\end{lemma} 

\begin{proof}
We apply Lemma~\ref{lemcode} with $\BB = L^2$  to the equation \eqref{odej}. 
With  $a _0 = 1 + \eps c'(\theta)$ and coefficients $a_{1, j}, a_{2, j} $ given by \eqref{aedoj}, the conditions 
\eqref{odeb0} and \eqref{odeb1} are satisfied for some $m = m(m_1)$ if  \eqref{bound0} holds. Thus we have an estimate  \eqref{estode2} for $\| \dt^j Y \|_{L^2}$ with source term $\mfs_j$, to which  
 contribute the three terms in the right hand side of  \eqref{odej}. 

\medskip

\noindent -- {\bf a)} The first contributor is  $\chi_j = - \dt^{j+1} q$. Thus, for  $ j + 4 \le n +1 $, one has 
\begin{equation*}
\sum_{0 \le l \le 3} \| \dt^l   \chi_j (t) \|^2 _{L^2}   \le    M .  
\end{equation*} 
  Moreover, the integrated terms are  $O (\delta^2 + m^2 \eps^2 \delta^2) t  M$.  Hence the total contribution of 
  $\chi_j$  satisfies 
$$
 \mfs_{\chi_j} (t)  \le C (1 + (\eps + \delta^2) t ) M. 
$$

\medskip

\noindent -- {\bf b) } The second contributor  is $ \psi_{(j)} = 2 \dt^j \big(  q (1 + \eps c') \dt \theta \big)  $. 
Expanding the derivatives, we obtain a sum of terms of the form \eqref{termgj} and   
Lemma ~\ref{GNET} implies that for $j + 2 \le n+1$, 
\begin{equation*}
\label{estpsij}
\eps \int_0^t \|   \psi_{(j)}(s)  , \dt \psi_{(j)}(s)  \|^2_{L^2} ds  
\le  \eps  C  ( m_1,  K_0)  \Big( K_0 +  \int_0^t \sum_{j=1}^{n+1}\Vert \dt^j U(s)\Vert_{L^2} {\rm d}s \Big)  
\end{equation*}
where $K_0$ is given by \eqref{M1}. 

In addition, we note that all the terms $ \dt^{l} \theta (0)$ for $l \le j+1 \le n $ belong 
to $H^1(\cE)$ with norm   at most  $ K_0 $. Therefore 
\begin{equation*}
\| \psi_{(j)} (0) \|_{L^2}  \le C(K_0 ) M. 
\end{equation*}
Hence
the total contribution of 
  $\psi_{(j)}$  satisfies 
$$
 \mfs_{\psi_j} (t)  \le C  (1 +M+ \eps  t  M). 
$$

\medskip
 
\noindent -- {\bf c) } The third  contributor    $ \varphi_{(j)}$  is a linear combination of terms 
$$
\varphi_* =  \delta      c^{(l+1)}  (\theta)  \theta_{j_1} \ldots \theta_{j_l} \  \dt^k  Y 
$$
with $k < j$, $l \ge 1$ and $j_1 + \ldots + j_l = j+2 -k$.  First we use Lemma~\ref{GNET} to 
control $L^p$ norms  of $\dt^k Y $. For $  k / (n-3) \le 2/ p  \le 1 $, there holds 
 \begin{equation*}
\begin{aligned}
\big\| \dt^k Y & \big\|_{L^p([0,T] \times \mathcal E)} \\
 & \le C  
\big( K_1+ \big\| Y\big\|_{L^\infty([0,T] \times \mathcal E)} \big)^{1 - \frac{2}{p}}
\big( K_1   +   \sum_{l \le n -3}  \big\| \dt^l Y\big\|_{L^2([0,T] \times \mathcal E)} \big) ^{\frac{2}{p}}. 
\end{aligned}
\end{equation*}
 Let us consider two cases $l\ge 2$ and then $l=1$ to control terms involving $\theta$.

\smallskip

\noindent {\it Case $l\ge 2$:} By Lemma~\ref{GNET} applied to $\dt \theta$, one has for  $ (j_l  -1)  / n  \le 2/ p_l   \le 1 $, 
 $$
 \big\| \dt^{j_l} \theta \big\|_{L^{p_l} ([0,T] \times \mathcal E)} \\
 \le C  
\big( K_0 + \big\| \dt  \theta \big\|_{L^\infty([0,T] \times \mathcal E)} \big)^{1 - \frac{2}{p_l}}
\big( K_0   +   \sum_{l \le n+1}  \big\| \dt^l\theta \big\|_{L^2([0,T] \times \mathcal E)} \big) ^{\frac{2}{p_l}} .  $$
Then recalling that $j\le n-3$ 
 $$
 \sum \frac{j_l-1 }{n} + \frac{k}{n-3} \le   \frac{j  -k   }{n } + \frac{k}{n-3}  \le 1 
 $$
so that one can find indices $p_l $ and $p$ such that 
$$  (j_l  -1)  / n  \le 2/ p_l   \le 1, \qquad  k / (n-3) \le 2/ p  \le 1 $$
and 
$$2 / p + \sum 2/ p_l = 1.$$ 
Thus,  
$$
\begin{aligned}
\| \varphi_*  \|^2_{L^2 ([0, t] \times \cE )} \le&  \delta^2 C (K_0, m_1) \Big(  K_1^2 + \int_0^t \mfE_1 (s) ds \Big)
 \\
 &+ 
\delta^2 C (K_0, m_1)   \big( K_1+  \| Y \|_{L^\infty([0,T] \times \mathcal E)})^2
\Big( K_0^2 + M t  \Big) .  
\end{aligned}
$$
 If $l = 1$,  it remains terms of the form 
 $$
 \varphi_* = \delta  c^{(2)} (\theta) \dt ^{j+2 -k} \theta  \, \dt^k Y . 
 $$ 
 We apply  Lemma~\ref{GNET} to  $\delta \dt^2 \theta$: for  $ (j-k)   / (n-1)   \le 2/ p'    \le 1 $, 
 $$
 \begin{aligned}
 \big\| \delta  \dt^{j+2-k} \theta \big\|_{L^{p'} ([0,T] \times \mathcal E)} 
 \le C  
\big( K_0 &+ \big\| \delta \dt ^2  \theta \big\|_{L^\infty([0,T] \times \mathcal E)} \big)^{1 - \frac{2}{p' }}\\
&\times \big( K_0   +   \sum_{l \le n+1}  \big\| \delta \dt^l\theta \big\|_{L^2([0,T] \times \mathcal E)} \big) ^{\frac{2}{p'}} . 
\end{aligned}
 $$
Remarking that 
 $$
 \frac{j - k}{n-1} + \frac{k}{n-3} \le \frac{j}{n-3} \le 1,  
 $$
 we can choose indices  $p$ and $p'$ such that    $  k / (n-3) \le 2/ p  \le 1 $, $ ( j -k )   / (n-1)  \le 2/ p_l   \le 1 $, 
and $2 / p + 2/ p' = 1$ and in this case we obtain that
 $$
\begin{aligned}
 \| \varphi_*  \|^2_{L^2 ([0, t] \times \cE )} \le&    C (m_1)  ( K_0^2 + m_1^2 ) \Big(  K_1^2 + \int_0^t \mfE_1 (s) {\rm d}s \Big)
 \\
 &+
 C ( m_1)   \big( K_1+  \| Y \|_{L^\infty([0,T] \times \mathcal E)})^2
\big( K_0^2  + M t    \big) .  
\end{aligned}
$$
 Summing up, this shows that the contribution of $\varphi_{(j)}$ satisfies
\begin{equation*}
\label{estphij}
\begin{aligned}
\mfs_{\varphi_j} (t) = \eps \int_0^t \| \varphi_{(j)}(s) &  \|^2_{L^2}  ds \le   \eps C (m_1, K_0 )  \Big(  K_1^2 + \int_0^t \mfE_1 (s) {\rm d}s \Big)
 \\
 &+ 
 \eps C ( m_1, K_0)   \big( K^2_1+  \| Y \|^2_{L^\infty([0,T] \times \mathcal E)})
\big( 1  +  M t\big ) .   
\end{aligned}
\end{equation*}
  Adding up, we have obtained an estimate of $\mfs_j (t)$.  In $ \mfs_{\varphi_j}$   appears   a term 
  $$
  \eps   C (m_1, K_0 )   \int_0^t \mfE_1 (s) {\rm d}s 
  $$
  which is absorbed to the left by Gronwall's lemma. The other terms  are all controlled by the right hand side of \eqref{blop}. 
 \end{proof} 
 

\section{Proof of the main Theorem}
\label{final} 

We prove here Theorem \ref{theonew}. Using Proposition \ref{estIDfin} instead of Proposition \ref{estID}, one easily deduces Corollary \ref{coromain}.
We consider initial data $(\theta^{\rm in}, q^{\rm in}) \in \HH^n$, $ n \ge 5$,   satisfying for some $M>0$,
\begin{equation}\label{condini}
\begin{aligned}
& \jump{q_0^{\rm in}} = 0, \qquad M^{-1} \le 1 + \eps c' (\theta^{\rm in}_0) \le M, 
\\
 &  \| \theta_0^{\rm in} \|_{H^{n+1} (\mathcal E)}  \le M ,
 \quad \|  ( q_0^{\rm in}, \delta \dx q_0^{\rm in} )  \|_{H^{n+1} (\mathcal E)}  \le M , 
\end{aligned}
\end{equation} 
and the compatibility condition~\eqref{cc1new} which we recall here for the reader's convenience,  
\begin{equation}
\label{cc1newfin} 
\begin{cases}
\big\lvert \jump{q_{j+1}^{\rm in}} \big\rvert  \le M \delta^{n-j-1/2} \\
\big\lvert  \alpha \av{q_{j+1}^{\rm in}} +\jump{\theta_{j}^{\rm in}} -\delta^2 \jump{\dx q_{j+1}^{\rm in}} \big\rvert   \le M \delta^{n-j-1/2} 
\end{cases}
 \qquad  \mathrm{for} \ \  0 \le j \le n -1.
\end{equation}
   By Proposition \ref{prop1}, we know that the solution $U $ belongs to $   C^\infty ([0, T^* [ ; H^{n+1}(\mathcal E) \times H^{n+2}(\mathcal E))$ on a maximal  interval of time $[0, T^* )$, but with $T^* > 0$, possibly small. The following proposition shows that this maximal interval of time is at least of order $O\big((\eps+\delta^2)^{-1}\big)$.

\begin{proposition}\label{estim}
Let $n\geq 5$. Given a constant  $M>0$, and an initial data $(\theta^{\rm in},q^{\rm in})\in \HH^n= H^{n+1} (\mathcal E)\times H^{n+2} (\mathcal E)$ satisfying \eqref{condini} and \eqref{cc1newfin}, there is $\tau = \tau (M)$ such that 
$T^*  \ge  T_* =  \tau / (\eps + \delta^2)$. Moreover, there are constants 
$\underline{M}_k $ ($k=1,2,3$) such that  for 
$0 \le t \le  T_* $ one has
\begin{align}
\mfE (t) \le \underline{M}_1, \quad \mfE_1 (t) \le \underline{M}_2, 
    \quad \| U(t), \dx U(t)\|_{L^\infty(\mathcal E)}   \le \underline{M}_3. 
\end{align}
where $\mfe$, $\mfE_1$ are respectively given by \eqref{mfE} and \eqref{mfE1}.

\end{proposition}

\begin{proof}  We proceed in several steps. 

\medskip

\noindent -- {\bf a) } Introduce  for $t < T^*  $ 
\begin{equation}
m (t)  =  \| \theta, \dt \theta,\delta\dt^2 \theta, q, \dt q \|_{L^\infty ([0, t ] \times \mathcal E)} . 
\end{equation}

\begin{lemma}
There is $\tau_1 = \tau_1 (M, \underline{m})$ such that if $t < T^*$, $ \eps  t \le \tau_1  $ and $m(t) \le \underline{m}$, then 
$$
\frac{1}{2}  M^{-1} \le 1 + \eps c' (\theta) \le  2  M \quad  {on } \  [0, t ] \times \cE. 
$$
\end{lemma}  
\begin{proof} By Taylor expansion, there is $C(\um)$ such that 
$$
| c' (\theta)  - c' (\theta^{\rm in} ) |  \le  C (\um) t . 
$$
\end{proof}

\medskip

\noindent -- {\bf b) } Consider the energy $\mfE(t)$ defined in \eqref{mfE}. 
By Proposition~\ref{estIDfin} and Proposition \ref{addc}, we know that there is a constant 
$M_1 (M) $ such that the initial energy  satisfies 
\begin{equation}
\label{b1}
\max\{K_0,\mfE(0)\}  \le M_1. 
\end{equation} 
where $K_0$ and $\mfE$ are respectively given by \eqref{M1} and \eqref{mfE}.
 Proposition~\ref{Estdtj} therefore
implies the following estimate.

\begin{lemma}
There are $M_2 = M_2(M) $ and  $\tau_2 = \tau_2 (M, \um)\le \tau_1 $ such that if $t < T^*$, $ \eps  t \le \tau_2  $ and $m(t) \le \um$, then 
\begin{equation}
\label{b2}
\sup_{0 \le s \le t} \mfE(s )  \le    M_2  
\end{equation} 
\end{lemma} 

The important point is that $\tau_2$ may depend on $\um$, but not $M_2$.

\medskip

\noindent -- {\bf c) } Introduce 
$$
\mfm (t)  =   \sum_{l \le 2} \| \dt^l \theta \|^2_{L^\infty([0, t] \times \cE)}  + 
\sum_{l \le 4} \| \dt^l  q \|^2_{L^\infty ([0, t] \times \cE)} .
$$

We apply Lemma~\ref{Yinfty} noticing that by Proposition \ref{estIDfin}, $\| \dx \theta (0),\delta \dt\dx \theta(0) \|_{L^\infty} $ is controlled by $M_1$ (choosing a larger $M_1$ if necessary in \eqref{b1}).
\begin{lemma}
There are $C = C(M)$, $M_3 = M_3(M) $ and  
$\tau_3 = \tau_3 (M, \um )\le \tau_2  $ such that if $t < T^*$, $ (\eps + \delta^2)   t \le \tau_3  $ and $m (t) \le \um$, then 
\begin{equation}
\label{b3}
  \| \dx\theta (t)  \|^2_{L^\infty(\cE) }  \le       M_3   +  C (1 + \tau_3)   \mfm (t) . 
  \end{equation} 
\end{lemma}
 
\medskip

\noindent --  {\bf d)}  We now apply Lemma~\ref{YL2} to bound the energy $\mfE_1$ defined at \eqref{mfE1}. The initial value 
$\mfE_1 (0) $, as well as the constant $K_1$ defined in \eqref{M1}  are controlled by $M_1$  (choosing a larger $M_1$ if necessary in \eqref{b1}). Using the bound \eqref{b2} and demanding for instance that 
$$
(1 + ( \eps + \delta^2) t \,c(\um)) M_2 \le  2M_2
$$
one obtains that there are  $M_4 = M_4(M)$ and $\tau_4  = \tau_4 (M, \um )$
such that  if $(\eps + \delta^2 )t \le \tau_4$, 
\begin{equation*}
\mfE_1(t)  \le   M_4  (1 + \| \dx \theta  \|_{L^\infty([0, t] \times \cE)}^2 )   . 
\end{equation*}
Combining with \eqref{b3},  this implies the following lemma. 
\begin{lemma}
There are  $M_4 = M_4(M) $ and  
$\tau_4 = \tau_4 (M, \um)\le \tau_3  $ such that if $t < T^*$, $ (\eps + \delta^2)   t \le \tau_4  $ and $m(t) \le \um$, then 
\begin{equation}
\label{711}
\mfE_1(t)  \le   M_4  \big(1 +   (1 + \tau_3) \mfm (t)  \big)   . 
\end{equation}

\end{lemma} 

\medskip

\noindent -- {\bf e) }  We note now that $\mfm (t)$ is controlled by $\mfE_1(t)$ and $M_2$. Indeed, 
Lemma~\ref{lemdxq} and the Sobolev imbedding $H^1(\cE)  \subset L^\infty (\cE)$ imply that
\begin{equation*}
\sum_{l \le 4} \| \dt^l  q \|^2_{L^\infty ([0, t] \times \cE)}  \le  8 M_2 + C(M_1) + 
C(\um, M_1) \eps^2 t^2 M_2 \le   M_5 (M)
\end{equation*}
if $\eps t \le \tau_5$ and  $\tau_5 (M, \um)$ is small enough. 
We now use the interpolation estimate 
\begin{equation*}
\|  \psi \|_{L^\infty(\cE)}^2 \le C \| \psi \|_{L^2(\cE)}  \|  \dx \psi  \|_{L^2(\cE)}, 
\end{equation*}
which applied to $\dt^l \theta$ for  $l \le 2 \le n-3   $, implies that 
\begin{equation*}
\sum_{l \le 2} \| \dt^l  \theta  \|^2_{L^\infty ([0, t] \times \cE)}  \le   C \mfE (t)^{1/2}   \mfE_1(t) ^{1/2} . 
\end{equation*}
Hence we have proved that 
\begin{equation*}
\mfm (t)  \le   M_5  + \kappa \mfE_1 (t)  + \frac{C}{\kappa} M_2 . 
\end{equation*}
Inserting this estimate in \eqref{711} and choosing $\kappa=\kappa (M_4)$ small, we have proved the following result.
\begin{lemma}
If $n \ge 5$, there are constants $M_6 =  M_6(M) $ and  $\tau_6 = \tau_6 (M, \um)\le \tau_5  $
 such that if $t < T^*$, $ (\eps + \delta^2) t \le \tau_6  $ and $m (t) \le \um$,
\begin{equation}
\label{b4} 
\sup_{0 \le s \le t} \mfE_1(t)  \le       M_6. 
\end{equation}
\end{lemma} 
Again, the important point is that $\tau_6$ may depend on $\um$, but not $M_6$.

\medskip

\noindent -- {\bf f) } 
To close the loop, we note that Sobolev's imbedding theorem  implies that  if $n \ge 5$ 
\begin{equation}
\label{sobol}
m (t) \le   C_S \sup_{0 \le s \le t}  \big( \mfE_1(s)+ \mfE_1(s)\big) \le C_S (M_2+M_6) = M_7. 
\end{equation}

\medskip

\noindent -- {\bf g) }{\sl End of the proof.}  In the previous steps,  we have constructed constants 
$M_1, \ldots, M_7$ which depend only on $M$ given in the assumptions. 
We note also that 
$$
m(0)   \le M_1 . 
$$
Increasing $M_7$ if necessary, we can assume that $M_7 > M_1$ and we  now  \emph{choose} 
$\um = \um(M)  $ such that 
$$
 M_7 < \um , 
$$
For instance we can choose $\um = 2 M_7$. 
For this $\um$, there are $\tau_7  \le \ldots \leq \tau_1$ such that the estimates \eqref{b1} \ldots \eqref{sobol} are satisfied 
for $t \le \min\{ T^*, \tau_7 / (\eps + \delta^2)\}$, as long as $m(t)  \le \um$, and then, by \eqref{sobol}
$m (t ) \le M_7 < \um$. Note that, given the choice of $\um$, 
$\tau = \tau_7 (M, \um)$ is a function $\tau (M)$. \\
Since $m (0 ) \le M_1 \le M_7 < \um$, this implies by continuity of 
$m(t)$ that 
\begin{equation}
m(t) \le \um \qquad \mathrm{ for \ all\ }  \ t  \le \min \{ T^*, \tau  / (\eps + \delta^2)\}. 
\end{equation}
In particular the estimates \eqref{b1} \ldots \eqref{b4} are satisfied if $  t  \le \min \{ T^*, \tau / (\eps + \delta^2)\}$. 
By Lemma~\ref{lemdxq}, we see that $\| \dx q \|_{L^\infty} $ is bounded on this interval. 
Moreover, the $H^1$-norm of  $q$, and thus its $L^\infty$ norm, is bounded by 
$\mfE (t)$ and $\mfE_1 (t)$. Therefore, there is 
$M_8 (M)$ such that for $  t  \le \min \{ T^*, \tau / (\eps + \delta^2)\}$, 
\begin{equation}
 \big\| \theta (t)  ,  q (t) , \dx q (t) ,  1/ (1+ \eps c'(\theta(t) ) )\big\|_{L^\infty(\mathcal E)} \le M_8 .
\end{equation} 
Therefore, the blow-up criterion of  Proposition~\ref{prop1} implies that 
$T^* >  \tau / (\eps + \delta^2)$ and  the proof of the Proposition~\ref{estim} is now complete.
\end{proof}
Theorem \ref{theonew} is then a direct consequence of Proposition \ref{prop1}.


\end{document}